\documentclass[12pt,leqno]{article}
\usepackage{graphicx, fullpage}
\usepackage{amsmath,amssymb,amsthm,amscd, bm}
\usepackage{fancyhdr}
\usepackage[mathscr]{eucal}
\usepackage{amsfonts}
\begin{document}
\parskip=6pt

\theoremstyle{plain}

\newtheorem {thm}{Theorem}[section]
\newtheorem {lem}[thm]{Lemma}
\newtheorem {cor}[thm]{Corollary}
\newtheorem {defn}[thm]{Definition}
\newtheorem {prop}[thm]{Proposition}
\newtheorem {ex}[thm]{Example}
\numberwithin{equation}{section}

\newcommand{\cF}{{\cal F}}
\newcommand{\cA}{{\cal A}}
\newcommand{\cC}{{\cal C}}
\newcommand{\cH}{{\cal H}}
\newcommand{\cU}{{\cal U}}
\newcommand{\cK}{{\cal K}}
\newcommand{\cM}{{\cal M}}
\newcommand{\cO}{{\cal O}}
\newcommand{\cE}{{\cal E}}
\newcommand{\bC}{\mathbb C}
\newcommand{\bP}{\mathbb P}
\newcommand{\bN}{\mathbb N}
\newcommand{\bA}{\mathbb A}
\newcommand{\bR}{\mathbb R}
\newcommand{\fg}{\mathfrak g}
\newcommand{\fh}{\mathfrak h}
\newcommand{\ft}{\mathfrak t}
\newcommand{\fu}{\mathfrak u}
\newcommand{\fv}{\mathfrak v}
\newcommand{\fw}{\mathfrak w}
\newcommand{\fx}{\mathfrak x}
\newcommand{\fy}{\mathfrak y}
\newcommand{\fz}{\mathfrak z}
\newcommand{\fB}{\mathfrak B}
\newcommand{\fC}{\mathfrak C}
\newcommand{\fD}{\mathfrak D}
\newcommand{\fO}{\mathfrak O}
\newcommand{\fT}{\mathfrak T}
\newcommand{\fU}{\mathfrak U}
\newcommand{\fV}{\mathfrak V}
\newcommand{\fW}{\mathfrak W}
\newcommand{\fX}{\mathfrak X}
\newcommand{\fY}{\mathfrak Y}
\newcommand{\fZ}{\mathfrak Z}
\newcommand{\var}{\varepsilon}
\renewcommand\qed{ }
\newcommand{\sgrad}{\text{sgrad} \,}
\newcommand{\grad}{\text{grad}\,}
\newcommand{\id}{\text{id}}
\newcommand{\ad}{\text{ad}}
\newcommand{\const}{\text{const}\,}
\newcommand{\Ker}{\text{Ker}\,}
\newcommand{\opartial}{\overline\partial}
\newcommand{\Ree}{\text{Re}\,}
\newcommand{\Imm}{\text{Im}\,}
\newcommand{\an}{\text{an}}
\newcommand{\tr}{\text{tr}}
\newcommand{\abt}{^\text{ab}}
\newcommand{\ABt}{^\text{AB}}
\newcommand{\Hom}{\text{Hom}\,}
\newcommand{\End}{\text{End}\,}
\newcommand{\GL}{\text{GL}}
\newcommand{\Iso}{\text{Iso}\,}
\newcommand{\Hm}{\text{Hom}}
\newcommand{\Is}{\text{Iso}}

\begin{titlepage}
\title{\bf Aron--Berner--type extension in complex Banach manifolds\thanks{ 2020 Mathematics subject classification 32K05, 46G20, 58D10}}
\author{L\'aszl\'o Lempert\\
 Department of  Mathematics\\
Purdue University\\West Lafayette, IN
47907-2067, USA}
\end{titlepage}
\date{}
\maketitle
\abstract
Let $S$ be a compact Hausdorff space and $X$ a complex manifold. We consider the space $C(S,X)$ of continuous maps $S\to X$, and prove that any bounded holomorphic function on this space can be continued to a holomorphic 
function, possibly multivalued, on a larger space $B(S,X)$ of Borel maps. As an application we prove two theorems about bounded holomorphic functions on $C(S,X)$, one reminiscent of the Monodromy Theorem, the other of Liouville's Theorem. 
\endabstract

\section*{Introduction}

One version of the Aron--Berner extension theorem considers a complex Banach space $\fX$ and for every bounded holomorphic function 
on its unit ball provides a holomorphic extension to the unit ball of the bidual (second dual) 
$\fX''\supset \fX$. After the initial, somewhat 
weaker result of Aron and Berner, the sharp form above was proved by Davie and Gamelin. Aron and Berner also proved a similar theorem with the ball in 
$\fX$ replaced by a general open subset of $\fX$, that was sharpened subsequently by Zalduendo \cite{AB78, DG89, Z90}.---It is well understood in infinite dimensional complex analysis that the result is tight in many ways. While linear forms on $\fX$ can be extended
 to an arbitrary superspace $\fY$, this is no longer generally true for holomorphic functions, not even for quadratic functions.
 Work of Dineen and Josefson gives that extension even to $\fX''$ may fail without the boundedness assumption. For example, a holomorphic function $f$ on the space $\fX=c_0$ of complex sequences tending to $0$ extends to a holomorphic function on 
 $\fX''\approx l^\infty$ if and only if $f$ is bounded on every bounded subset of $\fX$ \cite{D71, J78}.
 
 A natural question is whether similar extension is possible for complex Banach manifolds $\fX$ rather than open subsets of Banach spaces.
 Of course, before talking about extensions of holomorphic functions, one should first construct a manifold $\fX''$ to play the role of the bidual, apparently
 a hard problem. Dineen and Venkova addressed it when $\fX$ is a Riemann domain over a certain type of Banach space, and extended holomorphic functions to a Riemann domain lying above the bidual \cite{DV04}. But in this paper we do not attempt a corresponding construction of $\fX''$ in general, and instead produce extensions in a more specialized framework, of Banach manifolds that arise as mapping 
 spaces. 
 
Let us start with a finite dimensional connected complex manifold $X$ and a  compact Hausdorff space $S$. The space $C(S,X)$ of
continuous maps $\fx:S\to X$ has the natural structure of a complex Banach manifold, see \cite{L04,LS07} and section 2 here. 
(When working with $X$, we will use $x,y,$ etc. to denote points of $X$, and the corresponding fraktur fonts 
$\fx,\fy,\dots$ to denote mappings into $X$.) Whatever the bidual of $C(S,X)$ should be, it must have something to do with the bidual of the Banach space $C(S)$ of continuous functions
$S\to\bC$. This bidual has been described, in slightly different ways, first by Kakutani, then by Gordon, as the space of continuous functions on a related space $T$ \cite {G66, K41}; the book \cite {D$^+$16} gives yet another variant of the construction of $T$ and of the 
isomorphism $C(S)''\approx C(T)$. This suggests that the bidual of a mapping space $\fX=C(S,X)$ should be $\fX''=C(T,X)$.
It turns out that the associations $(S,X)\mapsto C(S,X)=\fX$ and $(S,X)\mapsto C(T,X)=\fX''$ are parts of  (bi)functors between which
there is a natural transformation that embeds $\fX$ as a complex submanifold into $\fX''$. The question then would be whether bounded holomorphic functions on $\fX$ extend to $\fX''$. 

This is still not the question we will answer in this paper, though. With an application in mind, that we will discuss shortly, we rather consider
extension to a smaller manifold
\begin{equation}
B(S,X)=\{\fx:S\to X \text{ is Borel}\mid \fx(S)\subset X\text{ has compact closure}\}.
\end{equation}
(A map $\fx:S\to X$ is Borel if preimages of Borel subsets of $X$ are Borel.) The space of Borel maps in (0.1)
is also a complex Banach manifold. This is the space to which holomorphic functions 
on $C(S,X)\subset B(S,X)$ are to be extended.

It turns out that unlike in the Aron--Berner extension theorem, instead of holomorphic functions it is better to work with possibly multivalued
holomorphic functions, that is, holomorphic germs that admit analytic continuations along paths. We will define the relevant notions in section
7. For the moment we first submit that  the Aron--Berner extension provides for any holomorphic function germ $\bf f$ on $C(S,X)$ at some
$\fx_1$ a holomorphic germ $\bf f\abt$ on $B(S,X)$ (see section 5). We will denote the germ of a complex manifold $M$ at
some $p\in M$ by $(M,p)$ and, accordingly, a germ of a function on
 $M$ at $p\in M$ as ${\bf f}:(M,p)\to \bC$. 
\begin{thm}
Let $\fx_1\in C(S, X)$. If a holomorphic germ ${\bf f}:\big(C(S,X),\fx_1\big)\to\bC$ can be analytically continued along any path in $C(S,X)$ starting at $\fx_1$, and these continuations are uniformly bounded, then ${\bf f}\abt$ can be analytically
continued along any path in $B(S,X)$ starting at $\fx_1$.
\end{thm}

Without the boundedness assumption the theorem would be false, already if $X=\bC$ and $S$ is either $[0,1]$, or it
consists of a convergent sequence and its limit. Interestingly, in the theorem even if the analytic continuations of $\bf f$ define a single valued 
function on $C(S,X)$, the continuations of $\bf f\abt$ may define a multi--valued function on $B(S,X)$. We will discuss this and related
examples in section 11.

Since $B(S,X)$ is connected, Theorem 0.1 has the paradoxical consequence that if the germ $\bf f$ has
uniformly bounded analytic continuations along paths in $C(S,X)$ starting at $\fx_1$, then it has {\sl canonical}
continuations to all of the possibly disconnected $C(S,X)$. 

Spaces of continuous maps are, of course, central objects in topology, while Borel maps occur rarely, if ever,
and for good reason: The topology of $B(S,X)$ is too directly related to the topology of $X$, see Lemma 10.2
(and (7.1), Lemma 9.4a). But this is precisely the reason why the connection between the two 
mapping spaces, described in Theorem 0.1, has implications about the continuous mapping spaces 
themselves:

\begin{thm}[Monodromy Theorem] 
Suppose $X$ is simply connected and $\fx_1\in C(S,X)$. Assume a 
holomorphic germ ${\bf f}:\big(C(S,X),\fx_1\big)\to\bC$ has analytic continuation along any path in $C(S,X)$ starting at $\fx_1$. If these 
continuations are uniformly bounded, then
\newline (a) they define a single valued holomorphic function on the component of $\fx_1$ in $C(S,X)$.
\newline (b) If, in addition, $X$ is compact, then this single valued function is in fact constant.
\end{thm}

We emphasize that $X$, and not as in the traditional Monodromy theorem the source manifold $C(S,X)$, is assumed to be simply connected.---Both theorems will be proved in greater generality in sections 9 and 10. 

That holomorphic functions (bounded or not) on certain mapping spaces must be constant for deeper reasons was first discovered by Dineen and Mellon \cite{DM98}. Subsequently with Szab\'o we generalized to spaces of `based' maps when $X$ is rationally connected, e.g., a
Grassmannian; if $S$ is a manifold, constancy follows even for spaces of $C^k$ maps \cite{L04, LS07}. Nonetheless, similarity with 
Theorem 0.2b is formal 
only, and the underlying mechanisms are different. That this must be so can be seen upon comparing what is covered here and there. On the one hand, the manifolds $X$ of Theorem 0.2 are vastly more general than the rationally connected manifolds of \cite{LS07}; on the other,
the requirement of boundedness and that we work with spaces of continuous, rather than $C^k$ maps, severely restricts the
scope of Theorem 0.2b, say. But this is inevitable. For general compact simply connected $X$ and a compact smooth manifold $S$ the
spaces $C^k(S,X)$ of $C^k$ maps may support nonconstant holomorphic functions when $k\ge 1$. For example, a Borel measure $\mu$ on the circle $S^1$ and a holomorphic section $\alpha$ of 
$(T^*X)^{\otimes m}$, not identically vanishing on multivectors of form 
$\xi\otimes\xi\otimes\dots\in (TX)^{\otimes m}$, induces a holomorphic function $a:C^1(S^1,X)\to\bC$,
$$
a(\fx)=\int_{S^1}\alpha(\dot\fx\otimes\dot\fx\otimes\cdots)\,d\mu,
$$
typically not even locally constant.

Related results have been proved by Aron, Galindo, Pinasco, and Zalduendo. In \cite{A$^+$16} they consider, for a compact space $S$, holomorphic functions on $C(S)$ that are invariant under the group of linear maps 
$C(S)\to C(S)$ induced by homeomorphisms $S\to S$.
Among other things they find that when $S$ is a circle or a closed disc, such functions must be constant. While their assumptions are very 
different from ours: here boundedness, there invariance, both proofs rely on the idea of Aron--Berner extension.

Contents. Section 1 discusses the notion of Borel vector bundles. These are needed to endow the space 
$B(S,X)$ with the structure of a complex Banach manifold. Sections 2 and 3 review and discuss basic properties of mapping spaces and of the Aron--Berner extension in Banach spaces.
Section 4 is about extension in local models of mapping spaces: from spaces of continuous sections of certain 
topological 
vector bundles to spaces of Borel sections. In particular, we prove a naturality property of this extension, 
Corollary 4.6. Section 5 then uses this naturality
to extend holomorphic germs ${\bf f}$ on $C(S,X)$ to holomorphic germs ${\bf f}\abt$ on $B(S,X)$. All of this is rather 
straightforward. 
Sections 6 and 7 discuss notions of holomorphic Banach bundles, flat norms, and analytic continuation. 
Theorem 7.5 generalizes Theorem 0.1 to sections of certain Banach bundles. Section 8
features a key result: The construction of a family of holomorphic maps into $X$ of  balls in finite
dimensional vector spaces.
These will be used in section 9 to reduce analytic continuation in $B(S,X)$ to analytic continuation in Banach spaces. 
The proofs of Theorems 0.1, 0.2, generalized, are given in sections 9 and 10. Section 11 contains examples that show that
the main theorems are optimal in certain ways.

During the preparation of this paper I have profited from discussion with Pablo Galindo.

Our basic reference to complex analysis in Banach spaces is Mujica's book \cite{M86}, and to infinite dimensional complex manifolds, 
\cite[Section 2]{L98}. The infinite dimensional complex manifolds that we will encounter are what \cite{L98} calls rectifiable, and modeled on
complex Banach spaces; all this means is that they are locally biholomorphic to open subsets of Banach spaces. We denote the set of holomorphic maps between complex manifolds $M,N$ by $\cO(M,N)$,
and by $\cO(M)$ when $N=\bC$.

\section{Borel vector bundles}

While continuous mapping spaces $C(S,X)$ are Banach manifolds modeled on spaces of continuous sections of topological
vector bundles, the local models of spaces $B(S,X)$ of Borel maps are spaces of sections of what we call Borel vector
bundles. In this section we develop the notion in the generality we will need. Related but different notions, such as measurable
fields and bundles of Banach, especially Hilbert, spaces have been introduced before.

Let $S$ be a topological space and $B(S)$ the algebra of bounded Borel functions $S\to\bC$. Consider a map $\pi:V\to S$
of sets, with each $V_s=\pi^{-1}s$, $s\in S$, endowed with the structure of a finite dimensional complex vector space. The
space of all sections of $\pi$, with pointwise addition and multiplication by bounded Borel functions, forms a $B(S)$--module.
By a Borel field of (finite dimensional) vector spaces we mean the datum of such $\pi:V\to S$ plus a $B(S)$--submodule
$\Xi$ of the module of all sections. As is customary, instead of $(\pi:V\to S,\Xi)$, 
often we just write $\pi:V\to S$, or even just $V$, for a Borel field of vector spaces, and $\Xi$ is only implied.
 If $\dim V_s$ is the same for all $s\in S$, we call this number the rank of $V$.

Suppose $T\subset S$ is Borel. We write $V|T$ first for $\pi^{-1}T$, but also for the Borel field 
$(\pi^{-1}T\to T,\Xi|T)$,
where $\Xi|T=\{\xi|T:\xi\in\Xi\}$. This is the restriction of $V$ to $T$.

A homomorphism of  Borel fields of vector spaces $(V\to S,\Xi)$, $(W\to S,H)$ is a map
$F:V\to W$, linear between $V_s$ and
$W_s$, $s\in S$, such that $F\circ\xi\in H$ for all $\xi\in\Xi$. A homomorphism is an isomorphism if it has 
an inverse, which
is also a homomorphism.

Any finite dimensional complex vector space $A$ defines a trivial Borel field of vector spaces $S\times A\to S$ and $\Xi$
consisting of sections that correspond to bounded Borel functions $S\to A$. Denote the space of linear maps between
vector spaces $A,B$ by $\Hom(A,B)$. Any bounded Borel map $f:S\to\Hom(A,B)$ induces a homorphism 
\[
F:S\times A\ni(s,a)\mapsto \big(s,f(s)a\big)\in S\times B
\]
of trivial Borel fields of vector spaces, and any homomorphism arises in this way.
\begin{defn} 
A Borel vector bundle is a Borel field of vector spaces that is isomorphic to a trivial Borel field of vector spaces.
\end{defn}
Admittedly, this does not have the ring of ``vector bundle'', that should be trivial only ``locally''. Still, if $S$ is the finite
union of Borel sets $S_j$, and for a Borel field of vector spaces $V\to S$ the restrictions $V|S_j$ are isomorphic to trivial
fields of the same rank, then $V$ itself is isomorphic to a trivial Borel field of vector spaces.

Any topological vector bundle $\pi:V\to S$ over a compact base, with equidimensional fibers, determines a Borel vector field
$(\pi:V\to S,\Xi)$: a section $\xi$ of $\pi$ is in $\Xi$ if in any topological trivialization over some compact $K\subset S$,
$V|K\approx K\times A$, it corresponds to a bounded Borel function $K\to A$.---Further examples of Borel vector
bundles can be obtained by pullback. Suppose $(\pi:V\to S,\Xi)$ is a Borel vector bundle, and $g:T\to S$ is a Borel
map of topological spaces. Then $g^*V=\coprod_{t\in T}V_{g(t)}$, with the obvious projection $\rho:g^*V\to T$,
and the
$B(T)$--submodule $H$ generated by sections of form $\eta(t)=\xi\big(g(t)\big)$, $\xi\in\Xi$, defines the pullback 
 Borel vector bundle $(g^*V\to T,H)$, also denoted $g^*V$. (Here we identify  fibers of the pullback of a
 bundle $V$ with the fibers of $V$ itself, something we will do elsewhere in the paper as well.)
 Equivalently, $H$ consists of those sections $\eta$ of $\rho$
for which the map $T\ni t\mapsto \eta(t)\in V_{g(t)}\subset V$ is bounded and Borel, these properties measured using any
trivialization $V\approx S\times A$. 

A norm on a trivial  Borel vector bundle $S\times A\to S$  is a Borel function $|\,\,|:S\times A\to[0,\infty)$ whose restrictions
to the fibers $\{s\}\times A$ define norms $|\,\,|_s$ on $A$, and there is a $c\in\bR$ such that
$|\,\,|_s\le c|\,\,|_t$ for all $s,t\in S$. Any isomorphism $F:S\times B\to S\times A$ pulls
back the norm $|\,\,|$ to a norm $|\,\,|\circ F$ on $S\times B\to S$. Therefore we can define the notion of a norm on a general
Borel vector bundle $V\to S$ as a function $V\to[0,\infty)$ that, under some (or an arbitrary) isomorphism
$V\approx S\times A$, corresponds to a norm on $S\times A\to S$. Such a norm is measurable with respect to the
Borel $\sigma$--algebra of $V$, defined as the pullback of the Borel $\sigma$--algebra
of $S\times A$ by any isomorphism $V\approx S\times A$. 

Given a Borel vector bundle $(V\to S,\Xi)$, it will be convenient to denote $\Xi$ 
by\footnote{That we use similar notation 
$B(S)$, $B(V)$---and later $C(S)$, $C(V)$---for rather different objects should not confuse the reader. In the former $S$ is a topological space and $B(S)$ the space bounded Borel functions on it; $V$ is Borel vector
bundle over $S$. Similarly, when $S$ is a compact Hausdorff space and $V\to S$ a topological
vector bundle, $C(S)$ is the space of continuous functions on $S$, and $C(V)$ is the space of continuous sections.} $B(V)$, and refer to it as the space of
bounded Borel sections of $V$. A norm $|\,\,|$ on $V$ turns $B(V)$ into a Banach space, with norm $||\xi||=\sup_S |\xi|$.
The topology on $B(V)$ is independent of the choice of $|\,\,|$. 

We finish this section by relating analysis on a Borel vector bundle $\pi:V\to S$ and on the Banach(able) space $B(V)$.
Given $U\subset V$, a norm $|\,\,|=p$ on $V$, and  $\var>0$, set
\[
U_{p,\var}=\{u\in U: p(u)<1/\var\text{ and }U\text{ contains all $v\in V_{\pi u}$ with }p(u-v)<\var\}.
\]
If $q$ is another norm on $V$, then $U_{p,\var}\subset U_{q,c\var}$ with some
$c>0$ independent of $\var$. Define 
\begin{equation} 
\hat U=\text{int}\,\{\xi\in B(V):\xi(S)\subset U\}=\bigcup_{\var>0}\{\xi\in B(V):\xi(S)\subset U_{p,\var}\}.
\end{equation}

Consider another Borel vector bundle $W\to S$ and a fiberwise $H:U\to W$, i.e., $H(U_s)\subset W_s$,
where $U_s=U\cap V_s$. Pick norms
$p,q$ on $V,W$.
\begin{defn}
We say that $H$ is of bounded type if $q$ is bounded on $H(U_{p,\var})$ for all $\var>0$.
\end{defn}
Clearly, bounded type are independent of the choice of $p,q$
\begin{lem}
Suppose the fibers $U_s\subset V_s$ are open. If  $H:U\to W$ is Borel, of bounded type, and 
$H|U_s$ are holomorphic for $s\in S$,  then
\[
H^\circ:\hat U\ni\xi\mapsto H\circ\xi\in B(W)
\]
is holomorphic.
\end{lem}
\begin{proof}
That $H^\circ$ maps into $B(W)$ follows from the definition in a straightforward way. To prove holomorphy,
we can assume that $V=S\times A$,
$W=S\times B$ are trivial.
We will write $d_AH$ for the differential of $H$ along $A$. Cauchy's formula along 
one dimensional slices in $\{s\}\times A$ gives that partials of 
$H$ along $A$ are Borel and bounded on $U_{p,\var}$. Using this for first derivatives gives that $H^\circ$
is continuous; for second derivatives, that  the directional derivatives 
\[
dH^\circ(\xi;\eta)=\lim_{\bR\ni t\to 0}\big(H^\circ(\xi+t\eta)-H^\circ(\xi)\big)/t,
\]
 exist for any $\xi\in \hat U$ and $\eta\in B(V)$, and $dH^\circ(\xi;\eta)(s)$ can be computed as 
$d_AH\big(\xi(s);\eta(s)\big)$. Now $S\times TA\to S$ and
$S\times TB\to S$ have a natural structure of a Borel vector bundle. Since the fiberwise tangent bundle
$T_AU=\coprod_{s\in S}T\big(U\cap(\{s\}\times A)\big)$ is a Borel subset of $S\times TA$,
$dH_A:T_AU\to S\times TW$ 
is Borel, of bounded type, and holomorphic along the fibers, $dH^\circ(\xi;\eta)$ is 
a continuous function of $\xi,\eta$. It is also complex linear  in $\eta$, which proves that $H^\circ$ is holomorphic.
 \end{proof}
 
\section{The mapping spaces}
 
Henceforward we work with a fixed finite dimensional, connected complex manifold $X$ and a compact Hausdorff space $S$. Write $TX\to X$ for the {\sl real} tangent bundle, on which the complex
structure of $X$ induces the structure of a holomorphic vector bundle. As in the Introduction, we let
$$
B(S,X)=\{\fx:S\to X \text{ is Borel}\mid \fx(S)\subset X\text{ has compact closure}\},
$$
and endow it with the uniform topology. Basic neighborhoods of $\fx\in B(S,X)$ are determined by 
open neighborhoods $D\subset X\times X$ of the diagonal $\{(x,x):x\in X\}$, the corresponding 
open neighborhood of $\fx$ being
\begin{equation} 
D(\fx)=\{\fy\in B(S,X):(\fx\times\fy)(S)\text{ has compact closure in } D\}.
\end{equation}
Here $(\fx\times \fy)(s)=\big(\fx(s),\fy(s)\big)$.

This topology is metrizable. Take any metric $d$ on $X$ that induces its topology; then the metric
\begin{equation}
d_S(\fx,\fy)=\sup\{d\big(\fx(s),\fy(s)\big):s\in S\}
\end{equation}
induces the uniform topology of $B(S,X)$. (The $\sup$ in (2.2) is finite, because $\fx(S),\fy(S)$ are contained in some compact
$K\subset X$ and $d$ is continuous on $K\times K$.) The topology restricts to the compact--open topology on the subspace
$C(S,X)$ of continuous maps.

The definition of a manifold structure on $B(S,X)$ involves the induced Borel vector bundles
\[
\fx^*TX=\coprod_{s\in S} T_{\fx(s)}X\to S,\qquad \fx\in B(S,X),
\]
see section 1, and their spaces $B(\fx^*TX)$ of bounded Borel sections. The material in section 1
allows to treat topological vector bundles only over compact bases as Borel vector bundles, and
then pull them back by Borel maps. Here in general we have to take any  compact  
$K\subset X$ containing $\fx(S)$; then
$TX|K$ determines a Borel vector bundle and can be pulled back by $\fx$ to produce a Borel 
vector bundle, that we denoted $\fx^*TX$.

If $D\subset X\times X$, $F$ is a function defined on $D$, and $x\in X$, we let 
\begin{equation}
D^x=\{y\in X: (x,y)\in D\} \quad\text{and}\quad F^x=F(x,\cdot).
\end{equation}
\begin{lem}
There are a neighborhood $D\subset X\times X$ of the diagonal and a $C^\infty$ diffeomorphism $F$ between $D$ and a neighborhood of the zero section in $TX$ with the following properties:
\newline (a) $F^x$ maps $D^x$ biholomorphically on a convex open subset of $T_xX$;
\newline (b) $F^x(x)\in T_xX$ is the zero vector;
\newline (c) Identifying the tangent spaces of the vector space $T_xX$ with $T_xX$ itself, the differential 
$F_*^x|_x: T_xD^x=T_xX\to T_xX$ is the identity.

Moreover, if $X$ is relatively compact and open in another complex manifold $M$, then $F$ can be
chosen real analytic.
\end{lem}
The inverse of $F$ is thus as good a substitute for a holomorphic exponential map $TX\to X\times X$ as it 
gets. Denoting by pr$_2:X\times X\to X$ projection on the second factor, accordingly we define
\begin{equation} 
\exp=\text{pr}_2\circ F^{-1}:TX\to X,
\end{equation}
whose restriction to $T_xX$ is a biholomorphism $F^x(D^x)\to D^x$, inverse to $F^x$. 

\begin{proof} 
The first part is \cite[Lemma 1.1]{LS07}. The second part is a special case of \cite[Lemma 2.1]{L04}, 
that (when $k=\omega$) takes a compact real analytic manifold $V$, possibly with boundary, a real analytic map
$\fx:V\to M$, and relates neighborhoods of the graph of $\fx$ in $V\times M$ and of the zero section
in the induced bundle $\fx^*TM$.---Note, though, the different use of $X$ here and there.---To apply 
that result, following Grauert \cite{G58} embed our $M$ in
some $\bR^N$ as a closed real analytic submanifold, choose $V\subset M$ a suitable sublevel
set of the Euclidean norm function, and $\fx$ the embedding of $V$ in $M$.
\end{proof}

In fact, a minor modification of the proof in \cite{L04} gives that $F$ can be chosen real analytic for {\sl any}
$X$, whether $M$ as above exists or not. 
 
We are now in the position to define holomorphic charts that turn $B(S,X)$ into a complex manifold. Given $D,F$ as in the lemma, 
a coordinate neighborhood will be $D(\fx)$ of (2.1). The map
\begin{equation}
\varphi_\fx:D(\fx)\ni\fz\mapsto F\circ(\fx\times\fz)\in B(\fx^*TX)
\end{equation}
is a homeomorphism on a neighborhood of the zero section in the Banach space $B(\fx^*TX)$. Since the inverse of $\varphi_\fx$ is given by $\xi\mapsto \exp\circ\xi$,
to show that the local coordinates (2.5) are holomorphically related we need to check that for $\fx,\fy\in B(S,X)$ the map
\begin{equation}
\varphi_\fy\circ\varphi_\fx^{-1}: \xi\mapsto F\circ\big(\fy\times(\exp\circ\xi)\big)
\end{equation}
is holomorphic wherever defined on $B(\fx^*TX)$. This is a special case of the following result, that we will need later on as well.
\begin{lem}
Let $Y$ be another finite dimensional complex manifold, $u\subset S\times X$ and $U\subset S\times TX$ open, 
\[
\psi:u\to Y\quad\text{and}\quad \Psi:U\to TY
\]
continuous. Assume that $\Psi$ holomorphically maps the fibers $U\cap(\{s\}\times T_xX)$ to 
$T_{\psi(s,x)}Y$. Fix $\fx\in B(S,X)$ such that the closure of $(\id_S\times\fx)(S)$ is in $u$, and 
let $\fU\subset B(\fx^*TX)$ consist of
those sections $\xi$ for which the closure of $(\id_S\times\xi)(S)$ is contained in $U$. Then 
$\fy=\psi\circ(\id_S\times\fx)\in B(S,Y)$; for $\xi\in \fU$ 
\begin{equation}
\eta=\Psi\circ(\id_S\times\xi)
\end{equation}
is a bounded Borel section of $\fy^*TY$; and the map
\[
\Psi^\circ:\fU\ni\xi\mapsto \Psi\circ(\id_S\times\xi)\in B(\fy^* TY)
\]
is holomorphic.
\end{lem}

It is of course possible that $\fU$ is empty, in which case the last statement is, well, vacuous.---The holomorphy of the map
(2.6) follows from Lemma 2.2 if we let $Y=X$, $\psi:S\times X\to X$ the projection, and 
$\Psi(s,v)=F\big(\pi v,\exp(v)\big)$, independently of $s$. 

\begin{proof}[Proof of Lemma 2.2] The lemma can be reduced to Lemma 1.3, but it is easier to
write out a direct proof, quite similar to the proof of Lemma 1.3.
That $\fy,\eta$ are Borel follows because the composition of Borel functions is Borel; they are 
also bounded, because $\psi,\Psi$ are bounded on the relatively compact sets $(\id_S\times\fx)(S)$, respectively, 
$(\id_S\times\xi)(S)$. Next, the  map $\Psi^\circ$ is continuous at any $\xi_0\in\fU$. This follows from the uniform continuity of $\Psi$ over a compact neighborhood $K\subset U$ of the closure of $(\id_S\times\xi_0)(S)$. 

To prove $\Psi^\circ$ is complex differentiable 
at $\xi_0$, we use the same $K$ and another compact neighborhood $L\subset\text{int}\, K$. Cauchy estimates then give that partial derivatives of $\Psi$ along the fibers of $TX$ are bounded on $L$. By Taylor's formula this implies that, denoting by $d_2\Psi$
the differential of $\Psi$ along the manifolds $\{s\}\times T_xX$, the derivative of $\Psi^\circ$
at $\xi_0$ in the direction $\xi\in B(\fx^*TX)$ 
\[
d\Psi^\circ(\xi_0;\xi)=\lim_{\bC\ni t\to 0}\big(\Psi^\circ(\xi_0+t\xi)-\Psi^\circ(\xi_0)\big)/t
\]
exists, and can be computed as
\[
d\Psi^\circ(\xi_0;\xi)(s)=d_2\Psi\big(s,\xi_0(s);\xi(s)\big),
\]
cf. (2.7). One now argues, as with $\Psi^\circ$ itself,  that $d\Psi^\circ(\xi_0;\cdot):B(\fx^*TX)\to B(\fy^*TY)$ is continuous. Hence
$\Psi^\circ$ is complex differentiable at $\xi_0$, and therefore holomorphic on $\fU$.
\end{proof}

As we saw, one consequence of Lemma 2.2 is that the  local coordinates $\varphi_\fx$, $\fx\in B(S,X)$, are holomorphically related, 
and endow $B(S,X)$ with the structure of a complex manifold, modeled on Banach spaces of form $B(\fx^*TX)$.

For example, if $S=\{s_1,\dots,s_m\}$ is finite, then
\[
B(S,X)\ni\fx\mapsto\big(\fx(s_1),\dots,\fx(s_m)\big)\in X^m
\]
is a biholomorphism. In general, the correspondence $(S,X)\mapsto B(S,X)$ is natural: a continuous map $a:T\to S$ and a 
holomorphic map $X\to Y$ induce, by composition, a holomorphic map $B(S,X)\to B(T,Y)$. More generally instead of a holomorphic 
map we can use a family of holomorphic maps $\Phi:S\times X\to Y$. We formulate the result only when $a:T=S\to S$ is the identity.
\begin{lem}
Let  $\Phi:S\times X\to Y$ be continuous. If for each $s\in S$ the map  $\Phi(s,\cdot)$ is holomorphic, then
\[
\Phi^\circ:B(S,X)\ni\fx\mapsto \Phi\circ(\id_S\times\fx)\in B(S,Y)
\]
is holomorphic.
\end{lem}
\begin{proof}
Let $D'\subset Y\times Y$ and $F':D'\to TY$ be as in Lemma 2.1, but with $Y$ rather than $X$. If 
$\varphi_\fx$ is a local coordinate on $B(S,X)$ as in (2.5), $\fy=\Phi^\circ(\fx)$, and $\varphi'_\fy$ is a local coordinate on
$B(S,Y)$ at $\fy$, we need to to check that the map 
\[
\varphi'_\fy\circ\Phi^\circ\circ\varphi_\fx^{-1}:
\xi\mapsto F'\circ\big(\fy\times\Phi\circ(\id_S\times\exp\circ\xi  )\big)\in B(\fy^*TY)
\]
is holomorphic where defined on $B(\fx^*TX)$. But this follows from Lemma 2.2 if we choose 
$\Psi(s,v)=F'\big(\pi v,\Phi(s,\exp v)\big)$.
\end{proof}

The complex structure on $B(S,X)$ was defined in terms of $D\subset X\times X$ and $F:D\to TX$. Since 
Lemma 2.3 applies with any choice of $D,F$ (and $D',F'$), it follows by taking $Y=X$ and $\Phi(s,x)=x$ that,
in the end, different choices of $D,F$ give rise to the same complex structure on $B(S,X)$.

Continuous maps $S\to X$ form a complex submanifold $C(S,X)\subset B(S,X)$, modeled on Banach spaces
$C(\fx^*TX)$ of continuous sections, $\fx\in C(S,X)$. It is closed as a subset of $B(S,X)$, but not a direct submanifold:
while locally the pair $C(S,X)\subset B(S,X)$ is biholomorphic to the pair $C(\fx^*TX)\subset B(\fx^*TX)$, the space
$C(\fx^*TX)$ will not in general have a closed complement in $B(\fx^*TX)$. It is this circumstance that makes even the
local holomorphic extension problem from $C(S,X)$ to $B(S,X)$ difficult (but solved by the Aron--Berner extension theorem).

An $\fx_0\in C(S,X)$ and a closed $A\subset S$ determine spaces of based maps
\begin{equation*}
\fB=B(S,X,A,\fx_0)\subset B(S,X)\quad\text{and}\quad \fC=C(S,X,A,\fx_0)\subset C(S,X)\cap\fB.
\end{equation*}
$\fB=\{\fx\in B(S,X):\fx|A=\fx_0|A\}$, and $\fC$ denotes the connected component of $\fx_0$ in 
$ \{\fx\in C(S,X):\fx|A=\fx_0|A\}$.
These are closed complex submanifolds of $B(S,X)$, resp.  $C(S,X)$, modeled on subspaces
\[
B_A(\fx^*TX)\subset B(\fx^*TX),\quad\text{respectively}\quad C_A(\fx^*TX)\subset C(\fx^*TX),
\]
of sections that vanish on $A$. 
It is the manifolds 
$\fB$ and $\fC\subset\fB$ that are the targets of our investigations.

\section{Aron--Berner extension}

Let $\fX,\fY$ be Banach spaces, $x_0\in\fX$, $r>0$, and 
$$
B=\{x\in\fX:||x-x_0||<r\},\qquad B''=\{\xi\in\fX'':||\xi-x_0||<r\}.
$$ 
In this section we review how to extend a bounded $f\in\cO(B,\fY)$ to a bounded function 
$f\ABt\in\cO(B'',\fY'')$, with the 
same bound. 

The procedure starts with a(n always continuous) $k$--linear form $T:\fX^{\oplus k}\to\bC$, $k=0,1,\dots$,
and constructs its $k$--linear Aron--Berner extension $\bar T:\fX''^{\oplus k}\to\bC$; the association
$T\mapsto\bar T$ is linear. When $k=0$, $T$ is constant, and the
extension is the same constant; when $k=1$, $\bar T:\fX''\to\bC$ is the second transpose of $T$. 
In general, to define 
$\bar T$ by recurrence, suppose $\bar T$ exists for $k$--linear operators $T$; then take a $(k+1)$--linear
$T$. For any choice of $x_1,\dots,x_k\in\fX$ the second transpose of $T(x_1,\dots,x_k,\cdot):\fX\to\bC$ is
a linear form $T''(x_1,\dots,x_k,\cdot):\fX''\to\bC$. For each $\xi\in\fX''$ 
$$
T_\xi=T''(\cdot,\dots,\cdot,\xi):\fX^{\oplus k}\to\bC
$$ 
is $k$--linear, and depends linearly on $\xi$. The extensions $\bar T_\xi:\fX''^{\oplus k}\to \bC$ then define
\begin{equation} 
\bar T(\xi_1,\dots,\xi_{k+1})=\bar T_{\xi_{k+1}}(\xi_1,\dots,\xi_k), \qquad \xi_j\in\fX''.
\end{equation}

More generally, if $T:\fX^{\oplus k}\to\fY$ is a $k$--linear operator, for each $l\in\fY'$ we consider the
$k$--linear form $T_l=lT$. The extensions constructed above, for each $\xi_j\in\fX''$ 
produce a linear form $l\mapsto \bar T_l(\xi_1,\dots,\xi_k)$ on $\fY'$, i.e., an element 
$\bar T(\xi_1,\dots,\xi_k)\in\fY''$. This defines the $k$--linear extension $\bar T:\fX''^{\oplus k}\to\fY''$ of $T$.
The standard cautionary remark here is that if the arguments of $T$ are permuted to produce a $k$--linear
operator $T_1$, then $\bar T_1$ is not in general the same as $\bar T$ with arguments correspondingly
permuted. 

The extension $\bar T:\fX''\to\fY''$ of a linear operator $T$ is continuous when both $\fX'',\fY''$ are endowed with the
weak* topology. The extension of multilinear operators is continuous only in a weaker sense. Let us say that a
function $F:\fX''^{\oplus k}\to\fY''$ is slightly continuous to mean that for any $j=1,\dots,k$, if 
$\xi_1,\dots,\xi_{j-1}\in\fX$ and $\xi_{j+1},\dots,\xi_k\in \fX''$ are fixed, then 
$F(\xi_1,\dots,\xi_j,\dots,\xi_k)\in \fY''$ depends continuously on $\xi_j\in\fX''$ in the weak* topology of
$\fX'',\fY''$. The following is not new:

\begin{lem}
The extension $\bar T:\fX''^{\oplus k}\to\fY''$ of a $k$--linear $T:\fX^{\oplus k}\to\fY$ is slightly continuous.
\end{lem}
\begin{proof}
It suffices to prove when $\fY=\bC$, that can be done by induction. When $k=0$, $\bar T$ is constant. When
$k=1$, $\bar T$ is the second transpose of $T$, and the claimed continuity follows directly from the 
definitions. Suppose the lemma holds for a certain $k\ge 1$ and let $T$ be $(k+1)$--linear. To show
that $\bar T$ is slightly continuous, we have to check a certain continuity property for all $j=1,\dots,k+1$.
When $j\le k$, this property follows from the induction hypothesis as $\bar T_{\xi_{k+1}}$ is slightly
continuous, cf. (3.1). If $j=k+1$, and $\xi_i=x_i\in\fX$, $i\le k$, then again
$\bar T(x_1,\dots,x_k,\xi_{k+1})=T_{\xi_{k+1}}(x_1,\dots,x_k)=T''(x_1,\dots,x_k,\xi_{k+1})$ depends continuously on $\xi_{k+1}\in\fX''$.
\end{proof}

Since $\fX$ is weak* dense in $\fX''$ (Goldstine's theorem, \cite[p. 424]{DS88}), it follows that
in the weak* topologies of $\fX'',\fY''$
\begin{equation} 
\bar T(\xi_1,\dots,\xi_k)=\lim_{x_1\to\xi_1}\lim_{x_2\to\xi_2}\dots\lim_{x_k\to\xi_k} T(x_1,\dots,x_k), \qquad 
\xi_1,\dots,\xi_k\in\fX'',
\end{equation}
which is often taken as the definition of $\bar T$, e.g., in \cite{AB78}.

We return to the construction of the Aron--Berner extension of $f$. The construction is compatible with
translations; for this reason we assume that the center of the balls $B,B''$ is $x_0=0$.
Any homogeneous polynomial
$P:\fX\to\fY$ can be represented as $P(x)=T(x,\dots,x)$ with a unique symmetric multilinear 
$T:\fX^{\oplus k}\to\fY$; we define $P\ABt(\xi)=\bar T(\xi,\dots,\xi)\in\fY''$, a homogeneous
polynomial on $\fX''$.
The correspondence $P\mapsto P\ABt$ can be extended by linearity to the space of all polynomials
$P:\fX\to\fY$. We need the fact that, using $\|\,\,\|$ for norm on $\fY$ and on its iterated dual spaces,
\begin{equation}
\sup_{B''}||P\ABt||=\sup_B||P||.
\end{equation}
This was proved by Davie and Gamelin \cite{DG89} when $\fY=\bC$. Carando in \cite{C99} observed that
(3.3) holds for homogeneous polynomials; that it is true in general easily follows from \cite{DG89}. Indeed,
assume that $\sup_B||P||=1$. If $l\in\fY'$ has norm $\le 1$, then $P_l=lP:\fX\to\bC$ is a polynomial,
\(
\sup_B|P_l|\le 1,\text{ and } P_l\ABt=lP\ABt.
\)
In the latter formula $l$ acts on $\fY''$ by transposition. By \cite{DG89}, then
\begin{equation}
\sup_{B''}|lP\ABt|\le 1\quad\text{for}\quad l\in\fY',\,\, ||l||\le 1.
\end{equation}
But by Goldstine's theorem \cite[p. 424]{DS88} the unit ball of $\fY'$ is weak* dense in the unit ball
of $\fY'''$, so that (3.4) holds for $l\in\fY'''$. Hence $\sup_{B''}||P\ABt||\le 1$. The opposite inequality being
obvious, (3.3) holds.

The following result, a slight improvement on \cite{AB78,DG89}, is known in the field:
\begin{thm}
Let $f\in\cO(B,\fY)$ be bounded, $f=\sum_{k=0}^\infty P_k$ its expansion in homogeneous polynomials,
$\deg P_k=k$. Then the series
\[
f\ABt=\sum_{k=0}^\infty P_k\ABt
\]
converges uniformly on concentric balls $B_\rho''\subset B''$ of radius $\rho<r$; it represents the Aron--Berner
extension $f\ABt\in\cO(B'',\fY'')$ of $f$. In addition,
\begin{equation}
\sup_{B''}||f\ABt||=\sup_B||f||.
\end{equation}
\end{thm}
\begin{proof}
$P_k$ is given by
\[
P_k(x)=\frac 1{2\pi}\int_0^{2\pi} f(e^{it}x)e^{-ikt}\,dt.
\]
This and (3.3) imply
\[
\sup_{B''}||P_k\ABt||=\sup_B||P_k||\le\sup_B||f||.
\]
 This in turn implies, in view of the homogeneity of $P_k\ABt$, that $\sum_kP_k\ABt$ converges uniformly on 
  $B''_\rho\subset B''$, $\rho<r$, hence represents a holomorphic function $f\ABt:B''\to\fY''$. 
  The estimate (3.3), with $P$ a partial sum, yields
 \[
 \sup_{B''_\rho}\Big\|\sum_{k=0}^mP_k\ABt\Big\|\le\sup_{B} \Big\|\sum_{k=0}^m P_k\Big\|.
 \]
 Letting first $m\to\infty$ then $\rho\to r$, (3.5) follows.
 \end{proof}

Of course, $f\ABt$ is not the unique holomorphic extension of $f$ to $B''$. But it is natural in a 
sense. To formulate a theorem, we need the notion of (Arens) regularity:
\begin{defn}
A Banach space $\fY$ is regular if for any Banach space $\fZ$ (or, equivalently, only for $\fZ=\bC$), 
the extensions of any bilinear operator $T:\fY\oplus\fY\to\fZ$ and of
$T_1(x,y)=T(y,x)$ satisfy $\bar T_1(\xi,\eta)=\bar T(\eta,\xi)$ for all $\xi,\eta\in\fY$.
\end{defn}

Induction then shows that for any $k$--linear $T:\fY^{\oplus k}\to\fZ$ and any permutation $\sigma$
of $\{1,\dots,k\}$, the operator $T_1(x_1,\dots,x_k)=T(x_{\sigma(1)},\dots,x_{\sigma(k)})$ satisfies
$\bar T_1(\xi_1,\dots,\xi_k)=\bar T(\xi_{\sigma(1)},\dots,\xi_{\sigma(k)})$. Since 
$\bar T(\cdot,\xi_2,\dots,\xi_k)$ is
always  continuous in the weak* topologies, in regular spaces $\bar T$ is separately continuous in all its 
variables. \"Ulger shows in \cite{U87} that any C* algebra is regular; later on we will need a minor 
generalization of this. 
\begin{thm} 
In addition to assumptions and notation above, let $y_0\in \fY$, $s>0$, and
\(
B_\fY=\{y\in\fY:||y-y_0||<s\}.
\)
Suppose $\fZ$ is a third Banach space, $f:B\to B_\fY$ is holomorphic, $f(x_0)=y_0$, and $g:B_\fY\to\fZ$
is bounded and holomorphic. Then $(g\circ f)\ABt=g\ABt\circ f\ABt$, provided $g$ is a first degree 
polynomial, or $\fY$ is regular.
\end{thm}

\begin{proof} We will prove under the assumption that $\fY$ is regular. The other case follows by similar, but 
much simplified calculations. By translations we can arrange that $x_0=0$ and $f(x_0)=y_0=0$.
First we assume that $f$ and $g$ are polynomials, $g$ homogeneous of degree $k$. This means that
there are symmetric multilinear operators $T:\fY^{\oplus k}\to\fZ$ and $U_i:\fX^{\oplus i}\to\fY$,
$i=0,1,\dots,h$, such that
\begin{gather}
g(y)=T(y,\dots,y),\qquad f(x)=\sum_{i=0}^h U_i(x,\dots,x),\qquad\text{and therefore}\nonumber\\
(g\circ f)(x)=T\big(f(x),\dots,f(x)\big)=
\sum T\big(U_{j_1}(x,\dots,x),U_{j_2}(x,\dots,x),\dots\big),
\end{gather}
summation over $j_1,\dots,j_k\le h$. The function (3.6) is a polynomial in $x$. Its $j$'th
degree component is obtained from the symmetrization of the $j$--linear operator
\begin{equation*}
S_j(x_1,\dots,x_j)=\sum_{j_1+\dots+j_k=j}T\big(U_{j_1}(x_1,\dots,x_{j_1}),U_{j_2}(x_{j_1+1},\dots,x_{j_1+j_2}),\dots\big)
\end{equation*}
by setting $x=x_1=\dots=x_j$. The symmetrization $S_{j,\text{sym}}(x_1,\dots,x_j)$ is
\begin{equation} 
\dfrac1{j!}\sum_\sigma \sum_{j_1+\dots+j_k=j}T\big(U_{j_1}(x_{\sigma(1)},\dots,x_{\sigma(j_1)}),
U_{j_2}(x_{\sigma(j_1+1)},\dots,x_{\sigma(j_1+j_2)}),\dots\big),
\end{equation} 
the first sum over all permutations $\sigma$ of $\{1,\dots,j\}$. Now 
\begin{equation}
\begin{gathered} 
g\ABt(\eta)=\bar T(\eta,\dots,\eta), \qquad f\ABt(\xi)=\sum_i\bar U_i(\xi,\dots,\xi), \qquad\text{and by (3.2)}\\
(g\circ f)\ABt(\xi)=\sum_j\lim_{x_1\to\xi}\lim_{x_2\to\xi}\dots\lim_{x_j\to\xi} S_{j,\text{sym}}(x_1,\dots,x_j).
 \end{gathered}
 \end{equation}
We compute this iterated limit for the generic term in (3.7). Since each $U_i$ is symmetric, we may restrict
attention to $\sigma$ such that $\sigma(1)<\sigma(2)<\dots<\sigma(j_1)$, 
$\sigma(j_1+1)<\dots<\sigma(j_1+j_2)$, and so on (shuffles). Then the arguments of $T$ in this generic term
tend to $\bar U_{j_1}(\xi,\dots,\xi)$, $\bar U_{j_2}(\xi,\dots,\xi)$, etc., and since $T$ is separately continuous,
the limit of the generic term is $\bar T\big(\bar U_{j_1}(\xi,\dots,\xi),\bar U_{j_2}(\xi,\dots\xi),\dots\big)$. 
In light of (3.8) and (3.7) therefore the limit in (3.8) is indeed 
$\bar T\big(\sum_0^h\bar U_i(\xi,\dots,\xi),\dots,\sum_0^h\bar U(\xi,\dots,\xi)\big)$.

Next we let $f$ be arbitrary, but keep  $g$ $k$--homogeneous. We expand 
$f$ in a homogeneous series. The partial sums $f_j$ of the series, polynomials, tend to 
$f$ as $j\to\infty$, uniformly on concentric balls $\subset B$ of radius $<r$. By (3.5) and by 
what we have proved, 
\[
(g\circ f)\ABt=\lim_{j\to\infty}(g\circ f_j)\ABt=\lim_{j\to\infty}g\ABt\circ f_j\ABt=g\ABt\circ f\ABt.
\]

Finally, a general bounded holomorphic $g:B_\fY\to\fZ$ can be expanded in a homogeneous series
$g=\sum_k g_k$, converging uniformly on concentric balls $\subset B_\fY$ of radius $<s$. Again using 
(3.5), on some neighborhood of $0\in B''$
\[
(g\circ f)\ABt=\lim_{k\to\infty}\big(\sum_{j=0}^k g_j\circ f)\ABt=\lim_{k\to\infty}\sum_{j=0}^kg_j\ABt\circ f\ABt=
g\ABt\circ f\ABt,
\]
and $(g\circ f)\ABt=g\ABt\circ f\ABt$ follows by analytic continuation.
\end{proof}

\section{Extensions from certain $C(S)$--modules} 

In this section we restrict ourselves to certain Banach spaces $\fX$ that are modules over $C(S)$. These modules arise
from a topological complex vector bundle $V\to S$, of finite rank, over the compact space $S$; they are the space 
$C(V)$ of continuous sections of $V$ and certain subspaces. We embed the space $B(V)$ of bounded Borel sections
into the bidual of $C(V)$, and study the restriction $f\abt$ to $B(V)$ (and subspaces) of the Aron--Berner extension 
$f\ABt$ from $C(V)$ 
(and subspaces). The extension $f\mapsto f\abt$ has a certain naturality property. Since the mapping spaces $\fB,\fC$ are 
modeled on modules studied here, the results of this section directly apply to provide local Aron--Berner--type 
extensions in those mapping spaces. 

By a norm on $V$ we mean a continuous function $|\,\,|:V\to[0,\infty)$ that restricts to a vector space norm on each fiber $V_s$, $s\in S$.
Given a closed $A\subset S$ and a norm on $V$, we consider the Banach spaces
\[
B_A(V)=\{\text{bounded Borel sections $\xi$ of }V: \xi|A=0\},\qquad ||\xi||=\sup_S|\xi|,
\]
and $C_A(V)\subset B_A(V)$ the subspace of continuous sections. Multiplication by functions turns these spaces into modules
over the Banach algebra $C(S)$.  Thus $C_\emptyset(V)=C(V)$, $B_\emptyset(V)=B(V)$.

The main point will be the construction of a canonical embedding $B_A(V)\to C_A(V)''$. We start by describing
the dual $C_A(V)'$ in terms of measures. Let $M(S\setminus A)$ denote the space of complex valued regular Borel measures on $S\setminus A$. This is also a $C(S)$--module: the product $fm$ of $f\in C(S)$
with a measure $m$ is given by
\[
(f m)(T)=\int_T f\,dm,\qquad T\subset S\setminus A\text{ Borel};
\]
thus $d(f m)/dm=f$.  
Let $C_0(S\setminus A)$ be
the space of continuous complex functions on the locally compact space $S\setminus A$ that vanish 
at infinity, and 
$C_A(S)\subset C(S)$ the ideal of functions that vanish on $A$. The map
\[
C_A(S)\ni f\mapsto f|(S\setminus A)\in C_0(S\setminus A)
\]
is an isometric isomorphism. Let furthermore $V^*\to S$ be the bundle dual to $V$, endowed with the dual norm $|\,\,|^*$, 
and $\langle\,,\rangle$  the duality pairing between $V^*$ and $V$. The tensor product
\[
C(V^*)\otimes_{C(S)} M(S\setminus A)
\]
is the $C(S)$--module freely generated by $C(V^*)\times M(S\setminus A)$, modulo the the submodule generated by
\begin{equation} 
(fg_1+g_2,m)-f(g_1,m)-(g_2,m),\quad (g,fm_1+m_2)-f(g,m_1)-(g,m_2),
\end{equation}
$f\in C(S)$, $g_1,g_2,g\in C(V^*)$, $m_1,m_2,m\in M(S\setminus A)$. Write $g\otimes m$ for the class of
$(g,m)$. 
\begin{lem} 
Associating with $t=\sum_1^kg_j\otimes m_j\in C(V^*)\otimes_{C(S)}M(S\setminus A)$ the linear form
\begin{equation} 
C_A(V)\ni\eta\mapsto\int_{S\setminus A}\sum_{j=1}^k\langle g_j,\eta\rangle\,dm_j\in\bC
\end{equation}
defines an isomorphism
\begin{equation} 
C(V^*)\otimes_{C(S)}M(S\setminus A)\overset{\approx}\longrightarrow C_A(V)'
\end{equation}
of $C(S)$--modules.
\end{lem}
\begin{proof}
To see that different representations of $t$ produce the same linear form (4.2) we have to check that (4.2) gives zero when the representations correspond to the
vectors in (4.1); and this is obvious. Thus (4.2) indeed defines a homomorphism (4.3). 

Next we verify that (4.3) is surjective. Let $l\in C_A(V)'$. Suppose $U\subset S$ is open and $V$ is trivial 
over a neighborhood 
of $\bar U$. Fix $\xi_1,\dots,\xi_r\in C(V)$ and $g_1,\dots,g_r\in C(V^*)$ that form dual frames of $V,V^*$ over  $U$.
If $\chi\in C(S)$ is supported in $U$, Riesz's theorem, generalized \cite{R87}, implies that the linear forms
\[
C_0(S\setminus A)\approx C_A(S)\ni f\mapsto l(\chi f\xi_i\big)\in\bC
\]
can be represented as $l(\chi f\xi_i)=\int_{S\setminus A}f\, dm_i$ with $m_i\in M(S\setminus A)$.
Hence
\[
l(\chi\eta)=\sum_{i=1}^r l\big(\chi\langle g_i,\eta\rangle\xi_i\big)=\sum_{i=1}^r\int_{S\setminus A}\langle g_i,\eta\rangle\,dm_i,
\qquad \eta\in C_A(V).
\]
If $\fU$ is a finite cover of $S$ consisting of $U$ as above, $\chi=\chi_U$ form a subordinate
partition of unity, and we choose 
$r=r^U$, $\xi_i=\xi_i^U$, $g_i=g_i^U$, $m_i=m_i^U$ accordingly, then
\[
l(\eta)=\sum_{U\in\fU}l(\chi_U\eta)=\int_{S\setminus A}\sum_{U\in\fU}\sum_{i=1}^{r^U}\langle g_i^U,\eta\rangle\,dm_i^U,\qquad \eta\in C_A(V).
\]
Thus (4.3) sends $\sum_U\sum_i g_i^U\otimes m_i^U$ to $l$, and is therefore surjective.

Finally, suppose that $t=\sum h_j\otimes m_j$ is in the kernel of
(4.3). Assume first that each $h_j$ is supported in an
open $U$, and there are $\xi_1,\dots,\xi_r\in C(V)$, $g_1,\dots,g_r\in C(V^*)$ that form dual frames over $U$. Then
each $h_j$ is in the submodule generated by the $g_i$, and $t$ can be written as $\sum_i g_i\otimes n_i$, with $n_i\in M(S\setminus A)$ supported in $U$. Thus for any $f_1,\dots,f_r\in C_A(S)$ and $\eta=\sum_i f_i\xi_i\in C_A(V)$
\[
\sum_i\int_{S\setminus A}f_i\,dn_i=\int_{S\setminus A}\sum_i\langle g_i,\eta\rangle\,dn_i=0.
\]
Hence each $n_i=0$ and $t=0$. Now with a general $ t=\sum_j h_j\otimes m_j$ in the kernel of (4.3) and with a 
sufficiently fine partition of unity $\chi_U$, $U\in\fU$, we have
\(
t=\sum_U\chi_Ut=0
\)
by what we have already proved, since $\chi_Ut $ is also in the kernel.
\end{proof}

Next we define a bilinear pairing $P:B_A(V)\times C_A(V)'\to\bC$. If $\eta\in B_A(V)$ and $l\in C_A(V)'$ corresponds, under (4.3), to $t=\sum_j g_j\otimes m_j\in C(V^*)\otimes_{C(S)}M(S\setminus A)$, set
\begin{equation} 
P(\eta,l)=\int_{S\setminus A}\sum_j\langle g_j,\eta\rangle\,dm_j.
\end{equation}
Again, different representations of $t$ give the same integral in (4.4).
\begin{lem} 
$B_A(V)\ni\eta\mapsto P(\eta,\cdot)\in C_A(V)''$ is an isometric embedding, an extension of the canonical embedding $C_A(V)\to C_A(V)''$.
\end{lem}
\begin{proof}
The statement about the extension is immediate from the definition of the canonical embedding; which embedding is isometric. Next, fix $\eta\in B_A(V)$ and  
$\sum_j g_j\otimes m_j\in C(V^*)\otimes_{C(S)}M(S\setminus A)$. Given $\varepsilon>0$, it follows from 
Lusin's theorem that there is a
$\xi\in C_A(V)$ such that
\[
\sum_j\sup_S|g_j|^*\int_{S\setminus A} |\xi-\eta|\,d|m_j|<\varepsilon\qquad\text{and}
\qquad \sup_S|\xi|\le\sup|\eta|.
\]
Hence $|P(\xi,l)-P(\eta,l)|=\big|\int_{S\setminus A}\sum_j\langle g_j,\xi-\eta\rangle\,dm_j\big|<\varepsilon$. 
Since $||P(\xi,\cdot)||=\sup_S|\xi|$, we obtain $||P(\eta,\cdot)||\le\sup_S|\eta|$.

To prove the opposite inequality, let $s\in S\setminus A$ and $\delta_s$ the Dirac measure at $s$. Choose a linear $\lambda: V_s\to\bC$ such that
\[
|\lambda(v)|\le |v|\quad\text{ for all }v\in V_s,\quad\text{ with equality if } v=\eta(s),
\]
and define $l(\xi)=\lambda\big(\xi(s)\big)$, $\xi\in C_A(V)$. Thus $l\in C(V)'$ corresponds under (4.3) to $g\otimes\delta_s$ with an arbitrary $g\in C_A(V^*)$ such that $g(s)=\lambda$. Hence $|P(\eta,l)|=|\eta(s)|$. Since $||l||\le 1$, we have $||P(\eta,\cdot)||\ge|\eta(s)|$, and this being so for any $s\in S\setminus A$, 
in fact $||P(\eta,\cdot)||\ge \sup_S|\eta|$.
This completes the proof.
\end{proof}

Identifying $B_A(V)$ with its image in $C_A(V)''$ allows us to define a version of the Aron--Berner extension 
from $C_A(V)$ to $B_A(V)$. If $\Gamma\subset B_A(V)$ is a ball with center in $C_A(V)$ 
and $f\in\cO\big(\Gamma\cap C_A(V),\fY\big)$ is bounded, the
extension $f\ABt$ discussed in section 3 can be restricted to $\Gamma$ to produce a holomorphic
\begin{equation}
f\abt:\Gamma\to\fY'',
\end{equation}
with the same sup norm. Similarly, if $T:C_A(V)^{\oplus k}\to\fY$ is $k$--linear, the restriction to 
$B_A(V)^{\oplus k}$ of its Aron--Berner extension $\bar T$ will be denoted $T\abt$. 
\begin{thm}
In addition to $V\to S$, consider another topological complex vector bundle 
$W\to S$, $\xi_0\in C_A(V)$, an open
neighborhood $U\subset V$ of ${\xi_0(S)}$, and a ball $\Gamma\subset B_A(V)$ about $\xi_0$
such that $\xi\in\Gamma$ implies $\overline{\xi(S)}\subset U$. Suppose a continuous $H:U\to W$ maps
fibers $U\cap V_s$ holomorphically into corresponding fibers $W_s$. Then the map
\[
H^\circ:\Gamma\ni\xi\mapsto H\circ\xi\in B(W)
\] 
is holomorphic.

If, furthermore, $H^\circ$ is bounded on $\Gamma\cap C_A(V)$, then
\begin{equation} 
\big(H^\circ |\Gamma\cap C_A(V)\big)\abt=H^\circ.
\end{equation}
\end{thm} 

To prove Theorem 4.3 we need a lemma.
\begin{lem}
Suppose $J:V^{\oplus k}\to W$ is a continuous fiber map, fiberwise $k$--linear. Define a $k$--linear operator 
$J^\circ:B_A(V)^{\oplus k}\ni(\eta_1,\dots,\eta_k)\mapsto J\circ(\eta_1\times\dots\times\eta_k)\in B_A(W)$. Then 
\begin{equation*}
(J^\circ|C_A(V)^{\oplus k})\abt=J^\circ.
\end{equation*}
\end{lem}
\begin{proof}
In this proof we will use the topologies on $B_A(V),B_A(W)$ inherited from the weak* topologies on
$C_A(V)'',C_A(W)''$. By construction, these are induced by the pairings 
$P$ in (4.4), respectively the corresponding $Q:B_A(W)\times C_A(W)'\to\bC$.
First we show that $J^\circ(\eta_1,\dots,\eta_k)\in B_A(W)$ depends continuously on each $\eta_j\in B_A(V)$.
Let's prove for $j=1$; fix $\eta_2,\dots,\eta_k$.  We need to show that 
$Q\big(J^\circ (\eta_1,\dots\,\eta_k),l\big)$, for any $l\in C_A(W)'$, depends continuously on $\eta_1$. The map
\[
K:V\ni v\mapsto J\big(v,\eta_2(s),\dots,\eta_k(s)\big)\in W,\qquad v\in V_s,\quad s\in S,
\]
is a bounded Borel homomorphism $V\to W$, with adjoint $K^*:W^*\to V^*$. If $l$ corresponds to
$\sum_i h_i\otimes m_i\in C(W^*)\otimes_{C(S)}M(S\setminus A)$, cf. (4.3), then
\begin{align*}
Q\big(J^\circ(\eta_1,\dots,\eta_k),l\big)=
\int_{S\setminus A}\sum_i\langle h_i,K\circ\eta_1\rangle\,dm_i
=\int_{S\setminus A}\sum_i\langle K^*\circ h_i,\eta_1\rangle\, dm_i.
\end{align*}
The last expression is $l_1(\eta_1)$ with a certain $l_1\in C_A(V)'$, and continuity follows.

In particular, $J^\circ$ is slightly continuous, as is $(J^\circ |C_A(V)^{\oplus k})\abt$ by Lemma 3.1. Since the 
two agree on $C_A(V)^{\oplus k}$, and $C_A(V)$ is dense in $B_A(V)$, they agree everywhere.
\end{proof}

\begin{proof}[Proof of Theorem 4.3]  By restricting to components of $S$, we can assume that
$V$ and $W$ have equidimensional fibers. Then Lemma 1.3 implies that $H^\circ$ even has a
holomorphic extension to a neighborhood of $\Gamma$ in $B(V)$. 

Next we turn to (4.6). By translations we can arrange that $\xi_0=0$ and 
$H^\circ(\xi_0)=0$. In a neighborhood of the zero section of $V$ expand $H$ in a uniformly convergent 
series $H=\sum_{k=0}^\infty H_k$, with fiberwise $k$--homogeneous $H_k:V\to W$. Each $H_k$ is 
induced by a symmetric $J:V^{\oplus k}\to W$ as in Lemma 4.4, $H_k(v)=J(v,\dots,v)$ for $v\in V$. By the
lemma and by (3.5), in a neighborhood of $0\in\Gamma$
\[
\big(H^\circ|\Gamma\cap C_A(V)\big)\abt=\sum_k \big(H_k^\circ|\Gamma\cap C_A(V)\big)\abt=
\sum_k H_k^\circ=H^\circ.
\]
By analytic continuation therefore $\big(H^\circ|\Gamma\cap C_A(V)\big)\abt=H^\circ$ on all of 
$\Gamma$.
\end{proof}

We return to the notion of regularity, Definition 3.3.
\begin{lem} 
The spaces $C_A(V)$ are regular.
\end{lem}
\begin{proof}
As a C* algebra, $C_A(S)$ is regular by \cite[Theorem 2.8]{U87}. (That theorem is about regularity of
an arbitrary bilinear operator $X\times X\to X$, when $X$ is a C* algebra. However, the justification
given there applies to any bilinear operator on $X$ with values in any Banach space.)
It follows that $C_A(S)^{\oplus q}$ is also regular 
for any $q\in\bN$, because if $S_q$ is the disjoint union of $q$ copies of $S$ and $A_q\subset S_q$ is
the disjoint union of $q$ copies of $A$, then $C_A(S)^{\oplus q}\approx C_{A_q}(S_q)$. Now
$C_A(V)$ is a quotient of $C_A(S)^{\oplus q}$ for some $q$. Indeed, if
$\xi_1,\dots,\xi_q\in C(V)$ span each fiber $V_s$, then
\[
C_A(S)^{\oplus q}\ni (f_1,\dots,f_q)\mapsto\sum_{i=1}^qf_i\xi_i\in C_A(V)
\]
is surjective. Since regularity passes down to quotients according to \"Ulger, \cite[Corollary 2.4]{U87},
$C_A(V)$ is indeed regular.
\end{proof}

Putting this, Theorem 3.4, and Theorem 4.3 together we obtain:
\begin{cor} 
Suppose that in Theorem 4.3 $H^\circ(\Gamma)$ is contained in a ball $\Delta\subset B_A(W)$, 
centered at $H^\circ(\xi_0)$. If $g\in\cO(\Delta\cap C_A(W),\fY)$ is bounded, then
\begin{equation}
(g\circ H^\circ)\abt=g\abt\circ H^\circ.
\end{equation}
 \end{cor}

 \section{Extending germs in mapping spaces} 
 
 The naturality property in Corollary 4.6 makes it possible to generalize the Aron--Berner extension from Banach 
spaces to Banach manifolds that are mapping spaces. Given $\fx_0\in C(S,X)$ and a closed $A\subset X$, as
in section 2, 
\(
\fB=\{\fx\in B(S,X):\fx|A=\fx_0|A\},
\)
and $\fC$ is  the connected component of $\fx_0$ in$ \{\fx\in C(S,X):\fx|A=\fx_0|A\}$.

At this point we are extending germs of holomorphic functions only. Suppose $\fx_1\in\fC$, 
$\fY$ is a Banach space, and ${\bf f}:(\fC,\fx_1)\to\fY$ is a holomorphic germ. Let $D\subset X\times X$ and
$F:D\to TX$ be as in Lemma 2.1. Suppose $\fx\in\fC$ is such that $\fx_1$ is contained in
\begin{equation}
D_A(\fx)=\{\fz\in\fB:(\fx\times\fz)(S)\text{ has compact closure in }D\}\subset D(\fx),
\end{equation}
cf. (2.1). Thus $D_A(\fx)$ is a coordinate neighborhood in $\fB$, with local coordinate the
restriction of $\varphi_\fx$ of (2.5),
\begin{equation}
\psi_\fx:D_A(\fx)\ni\fz\mapsto F\circ(\fx\times\fz)\subset B_A(\fx^*TX).
\end{equation}

Say, $\bf f$ is the germ of a bounded holomorphic function $f:V\to\fY$. The pullback 
$f\circ\psi_\fx^{-1}$ is holomorphic in a neighborhood of $\psi_\fx(\fx_1)\in C_A(\fx^* TX)$; we can chose this 
neighborhood to be of form $\Gamma\cap C_A(\fx^*TX)$ with $\Gamma\subset B_A(\fx^* TX)$ a ball
centered at $\psi_\fx(\fx_1)$. As we saw in section 4, $f\circ\psi_\fx^{-1}$ then has an Aron--Berner type 
extension to a holomorphic $(f\circ\psi_\fx^{-1})\abt:\Gamma\to\fY''$. The germ of $(f\circ\psi_\fx^{-1})\abt$ at 
$\psi_\fx(\fx_1)$ is independent of which ball $\Gamma$ is chosen.
\begin{lem}
The germ of $(f\circ\psi_\fx^{-1})\abt\circ\psi_\fx$ at $\fx_1$ is independent of the choice of $\fx$.
\end{lem}

\begin{proof}
Suppose $\fx_1\in D_A(\fy)$ with some $\fy\in\fC$, and $g=f\circ\psi_\fy^{-1}$ is bounded and holomorphic on 
$\Delta\cap C_A(\fy^*TX)$, where $\Delta\subset B_A(\fy^*TX)$ is a ball about $\psi_\fy(\fx_1)$. The composition
$h=\psi_\fy\circ\psi_\fx^{-1}$ is holomorphic in a neighborhood of $\psi_\fx(\fx_1)\in B_A(\fx^*TX)$; 
we can choose $\Gamma$ above contained in this neighborhood, and so that $h(\Gamma)\subset \Delta$.
(5.2) implies
\[
h(\xi)=F\circ\big(\fy\times(\exp\circ\xi)\big),\qquad \xi\in\Gamma,
\]
see (2.4). Denote the bundle projection $\fx^*TX\to S$ by $p$, and for $v\in\fx^*TX$ in a neighborhood $U$ of 
$\psi_\fx(\fx_1)(S)$ set
\[
H(v)=F\big(\fy(pv),\exp(v)\big)\in\fy^* TX.
\]
This $H$ is continuous and fiberwise holomorphic, as in Theorem 4.3, and $H^\circ$ of that theorem agrees with $h$
(on a sufficiently small $\Gamma$). By Corollary 4.6, on $\Gamma$
\[
(f\circ\psi_\fy^{-1})\abt\circ\psi_\fy\circ\psi_\fx^{-1}=g\abt\circ H^\circ=
(g\circ h)\abt=(f\circ\psi_\fx^{-1})\abt,
\]
and indeed $(f\circ\psi_\fy^{-1})\abt\circ\psi_\fy=(f\circ\psi_\fx^{-1})\abt\circ\psi_\fx$ in a neighborhood
of $\fx_1$.
\end{proof}

Thus we can make the following definition:
\begin{defn}
If ${\bf f}:(\fC,\fx_1)\to\fY$ is a holomophic germ, then ${\bf f}\abt:(\fB,\fx_1)\to\fY''$  stands for its extension
$({\bf f}\circ\psi_\fx^{-1})\abt\circ\psi_\fx$. Here $\fx\in\fC$ is arbitrary as long as  $\fx_1\in D_A(\fx)$, 
and $\psi_\fx$
is the local coordinate (5.2).
\end{defn}

\section{Banach bundles and flat norms} 

In this section we introduce/review the notions in the title. The main result, Theorem 6.3, is that over a 
simply connected Banach manifold a holomorphic Banach bundle with a flat norm is trivial. Norms on
vector bundles are also called Finsler metrics; we have already discussed them for finite rank bundles.
We will not need the notion of general norms here, and will define only flat norms.
\begin{defn} 
(a) A holomorphic Banach bundle is a holomorphic map $\pi:E\to M$ of complex Banach manifolds, with each
fiber $E_x=\pi^{-1}x$, $x\in M$, endowed with the structure of a complex vector space. It is required that
for each $x\in M$ there be a neighborhood $U\subset M$, a complex Banach space $(W,\|\,\,\|)$, and a
biholomorphism $g:\pi^{-1}U\to U\times W$ (local trivialization), that for all $y\in U$ maps $E_y$ linearly
to $\{y\}\times W$.

(b) A flat norm on $E\to M$ is a function $p:E\to[0,\infty)$ such that the local trivializations $g$ above can
be chosen to satisfy
\[
p(e)=||w||\qquad\text{if}\quad e\in\pi^{-1}U,\quad g(e)=(y,w).
\]
\end{defn}
To describe the structure of flat normed holomorphic Banach bundles we need the following result.
If $W$ is a Banach space, $\End W$ is the space of continuous linear operators
$W\to W$, endowed with the operator norm, and $\GL(W)\subset\End W$ the set of invertibles. 
\begin{lem} 
Let $U$ be a connected open subset of a complex Banach space $Z$, $(W,||\,\,||)$ another complex Banach 
space, and $F:U\to\GL(W)$ holomorphic. If $||F(\zeta)w||=||w|| $ for all $\zeta\in U$ and $w\in W$, then
$F$ is constant.
\end{lem} 
\begin{proof}
It suffices to prove when $\dim Z=1$, the general case then follows by restricting to one dimensional slices.
So, let $Z=\bC$. Denoting derivative by a dot, we will show that $\dot F(\zeta)=0$ for arbitrary $\zeta\in U$.
To simplify notation, we assume $\zeta=0$, and denote $F(0)=P$, $\dot F(0)=2Q$. Thus
$F(\zeta)=P+2Q\zeta+\dots$ for small $\zeta\in \bC$. Let
\[
K(\tau)=1+\cos 2\pi\tau=1+(e^{2\pi i\tau}+e^{-2\pi i\tau})/2\ge 0,\qquad\tau\in\bR,
\]
be the first Fej\'er kernel. For $w\in W$ and small $\zeta$
\[
||(P+\zeta Q)w||=\big\|\int_0^1 K(\tau)F( e^{2\pi i\tau}\zeta)w\, d\tau\big\|\le 
\int_0^1 K(\tau)||F( e^{2\pi i\tau}\zeta)w||\, d\tau=||w||=||Pw||.
\]
Since the left hand side is a subharmonic function of $\zeta$, the maximum principle implies
$||(P+\zeta Q)w||=||w||$, when $|\zeta|<r$ with some $r>0$.

Denote the norm on the dual space $W'$ also by $||\,\,||$, and the transpose of $P,Q$ by $P',Q'\in\End W'$;
thus  $P'$ is an isometry. Furthermore
\begin{equation} 
||(P'+\zeta Q')u||=||u||,\qquad |\zeta|<r,\, u\in W',
\end{equation}
as well. Indeed, with $\langle\,,\rangle$ the pairing between $W'$ and $W$,
\[
||(P'+\zeta Q')u||=\sup_{||w||=1}\big|\langle (P'+\zeta Q')u,w\rangle\big|=
\sup_{||w||=1}\big|\langle u,(P+\zeta Q)w\rangle\big|\le ||u||=||P'u||.
\]
Again, the left hand side is a subharmonic function of $\zeta$, and (6.1) follows.

Suppose $u$ is an extreme point of the closed unit ball $\Delta\subset W'$. Then
$P'u\in \Delta$ is also an extreme point. Comparing this with (6.1), we conclude $Q'u=0$. Since in the
weak* topology $\Delta$ is compact, and by the Krein--Milman theorem it is the closed convex hull of its extreme points, $Q'v=0$ follows for all $v\in \Delta$. Hence $\dot F(0)=2Q=0$.
\end{proof}
\begin{thm} 
Suppose $\pi:E\to M$ is a holomorphic Banach bundle over a connected, simply connected base, and
$p$ is a flat norm on $E$. Then $(E,p)$ is globally trivial: there is a Banach space $(W,||\,\,||)$ and a
biholomorphism $g:E\to M\times W$  that maps each $E_x$ linearly to $\{x\}\times W$, and satisfies
\[
p(e)=||w||\qquad\text{if}\quad e\in E,\, g(e)=(x,w).
\]
\end{thm}
This is a standard consequence of Lemma 6.2. The proof depends on the following notion:
\begin{defn} 
Suppose $U\subset M$ is connected open and $E|U$ admits a trivialization 
$g:E|U\to U\times W$ as in Definition 6.1b. A section
$h$ of $E|U$ is called horizontal if $g\circ h:U\to U\times W$ is a constant section. (Lemma 6.2 implies that
horizontality of $h$ is independent of the choice of the local trivialization $g$.) A section $h$ over a general
open subset of $M$ is called horizontal if it is locally such.
\end{defn}
\begin{proof}[Proof of Theorem 6.3]
Define a Hausdorff topology on $E$ by declaring that a neighborhood basis of $e\in E$ consists of sets
of form $h(U)$, where $U\subset M$ is a neighborhood of $\pi (e)$ and $h$ is a horizontal section of
$E|U$. In this topology $\pi:E\to M$ is a covering, which must be trivial since $M$ is connected and simply
connected. This means that through every $e\in E$ there passes a unique global horizontal section $h_e$. Fix
an $a\in M$ and let $(W,||\,||)=(E_a,p|E_a)$. This is a Banach space, and
\[
g:E\ni e\mapsto \big(\pi(e),h_e(a)\big)\in M\times W
\]
is the global trivialization sought.
\end{proof}

If $E\to M$ is a holomorphic Banach bundle with a flat norm $p$ over a connected base,
the biduals of the fibers $E_x$ constitute a holomorphic Banach bundle $E''\to M$, on which the biduals of
the norms $p|E_x$ define a flat norm $p''$. Local trivializations $g$ of $(E,p)$ as in Definition 6.1b give
rise to the corresponding local trivializations $g''$ of $(E'',p'')$: $g''(y)$ is just the second transpose of $g(y)$,
$y\in U$.

These considerations lead to the Aron--Berner
extension of germs of holomorphic sections of holomorphic Banach bundles. Suppose that $\fX$ is a Banach
space, $B\subset\fX$ is a ball centered at 
$x\in\fX$,  $B''\subset \fX''$ is the bidual ball, and $(E,p)\to B''$ is a holomorphic Banach
bundle with a flat norm. Under the trivialization $g$ of Theorem 6.3, a bounded holomorphic section $f$ of 
$E|B$ induces
a bounded holomorphic section $g\circ f$ of $B\times W\to B$, i.e., a bounded $\phi\in\cO(B, W)$. The 
corresponding trivialization $g''$ of $(E'',p'')$ and the Aron--Berner
extension $\phi\ABt\in\cO(B'',W'')$, in turn, induce a bounded holomorphic section $f\ABt$ 
of $(E'',p'')$. Lemma 6.2 implies that any another norm respecting trivialization $E\to B''\times W$ is of form 
$g_1=(\id_{B''}\times T)\circ g$, with $T\in \End W$. This means that the corresponding trivialization of $E''$
is $g''_1=(\id_{B''}\times T'')\circ g''$. Since the corresponding $\phi_1=T\circ\phi$, by Theorem 3.4 
$\phi_1\ABt=T''\circ\phi\ABt$; hence $f\ABt$ is independent of the choice of $g$.

More generally, given a flat holomorphic Banach bundle $(E,p)$ over some open set in $\fX''$, and a germ
$\bf f$ of a local holomorphic section of $E|\fX$ at some $x\in\fX$, the same prescription defines 
the extension ${\bf f}\ABt$, a germ of a section of $E''$. One can even consider bundles over more general 
manifolds. Combining the above construction with the discussion of mapping spaces $\fC\subset\fB$ in
section 5, given a flat holomorphic Banach bundle $(E,p)\to\fB$, with any $E|\fC$ valued holomorphic germ
$\bf f$ at some $\fx\in\fC$ we can associate its Aron--Berner--type extension ${\bf f}\abt$, a holomorphic
germ at $\fx$, valued in $E''$. 

 \section{Analytic continuation and extension} 
 
 Consider a holomorphic Banach bundle $E\to M$ over a connected base, 
 and let $\cE\to M$ denote the sheaf of $E$ valued  
 holomorphic germs. Suppose $a<b$ are real numbers and $r:[a,b]\to M$ a path (that is, a continuous map). 
 \begin{defn}
 Given a germ ${\bf f}\in\cE$ at $r(a)$, its analytic continuation along the path $r$ is a continuous lift
 $[a,b]\ni t\mapsto {\bf f}_t\in\cE_{r(t)}$ of $r$ such that ${\bf f}_a=\bf f$.
 \end{defn}
Recall that continuity at $t\in[a,b]$ means that for all
$\tau\in[a,b]$ sufficiently close to $t$ the ${\bf f}_\tau$ are germs of the same local section of $E$
about $r(t)$. For this reason analytic continuation can be defined by a partition
$a=t_0<t_1<\dots<t_k=b$, open neighborhoods $U_j\subset M$ of $r[t_{j-1},t_j]$, $j=1,\dots,k$, and holomorphic sections $g_j$ of $E|U_j$ such that $g_j=g_{j+1}$ on a neighborhood of $r(t_j)$, and $\bf f$
is the germ of $g_1$ at $r(a)$.

This shows that analytic continuation along $r$ is unique; and that if the path is slightly perturbed, keeping its
endpoints fixed, then the continuation of $\bf f$ at $r(b)$ is not going to change. It follows that if a germ 
${\bf f}\in\cE_x$ admits
analytic continuation along every path in $M$ starting at $x$, then the analytic continuation at $y\in M$ 
depends only on the homotopy class of the path connecting  $x$ and $y$ along which we continue. In 
particular, if $M$ is also simply connected, then the analytic continuations define a holomorphic section
$f$ of $E$: $f(y)\in E_y$ will be the value at $y$ of the analytic continuation from $x$ to $y$.

Later on we will need the following. 

\begin{lem}
Suppose $P:M\to N$ is a holomorphic  and open map of complex Banach manifolds, $F\to N$ is a holomorphic Banach 
bundle, $[a,b]\ni t\mapsto x_t\in M$ is a path, and ${\bf g}_t$ are $F$ valued holomorphic germs at $P(x_t)\in N$. Let
${\bf k}_t$ be the pullback of ${\bf g}_t$ along $P$ to $x_t$, holomorphic $P^*F$ valued germs. If ${\bf k}_t$ 
depend continuously on $t$, then so do ${\bf g}_t$.
\end{lem}

\begin{proof}
To prove continuity at some $t\in[a,b]$, let $V\subset N$ be a neighborhood of $P(x_t)$ and $g$ a holomorphic section 
of $F|V$ whose germ at $P(x_t)$ is ${\bf g}_t$. Thus $g\circ P$ is a holomorphic section of $P^*F|P^{-1}V$ whose
germ at $x_t$ is ${\bf k}_t$. For $\tau\in[a,b]$ close to $t$, ${\bf k}_\tau$ is the germ of $g\circ P$ at $x_\tau$. As $P$
is open, this implies that the germ of $g$ at $P(x_\tau)$ is ${\bf g}_\tau$; which shows that $t\mapsto g_t$ is indeed continuous.
\end{proof}

In what follows, when we speak of analytic continuation of a germ ${\bf f}\in\cE_x$ along an arbitrary
path, we of course mean paths that start at $x$. Let $p$ be a flat norm on $E$. 
\begin{defn} 
We say that a germ ${\bf f}\in\cE_x$ admits bounded analytic continuation along every path if first, it
admits analytic continuations $t\mapsto{\bf f}_t\in \cE_{r(t)}$ along every path $r$ and second, there is a bound 
$C$ such that for all such continuations $p\big({\bf f}_t(r(t))\big) \le C$.
\end{defn}
If $M$ is simply connected, bounded analytic continuation along every path simply means that $\bf f$
extends to a holomorphic section $f$ of $E$ and $p(f)$ is bounded.

Let us return to mapping spaces of finite dimensional connected complex manifolds $X$. As earlier, given a
compact Hausdorff space $S$, a closed $A\subset S$, and a continuous $\fx_0:S\to X$, we consider the
spaces
\begin{equation} 
\fB=B(S,X,A,\fx_0)\supset\fC=C(S,X,A,\fx_0);
\end{equation}
$\fB=\{\fx\in B(S,X):\fx|A=\fx_0|A\}$ and $\fC$ is the connected component of $\fx_0$ in 
$\fB\cap C(S,X)$.
Let furthermore $E\to\fB$ be a holomorphic Banach bundle.  In section 6 we  
associated with an $E|\fC$ valued holomorphic germ $\bf f$ at some $\fy\in \fC$ 
its Aron--Berner type extension ${\bf f}\abt$, an $E''$ valued germ at $\fy$.
\begin{lem}
If $t\mapsto{\bf f}_t$ is the analytic continuation of $\bf f$ along a path in $\fC$, then $t\mapsto{\bf f}_t\abt$ 
is the continuation of ${\bf f}\abt$ along the same path.
\end{lem}
\begin{proof}
It suffices to prove when the path is contained in a coordinate neighborhood. Then we might as well
assume that $\fC$ is a ball in a Banach space $\fX$, $\fB\subset \fX''$ is a ball in a subspace,
$E\to\fB$ is trivial, and so $E$ and $E''$ valued germs $\bf f$ correspond to germs 
$\bf g$ valued in fixed Banach spaces $\fY,\fY''$. Now a special case of a result of Zalduendo, that concerns
the Aron--Berner extension from a Banach space to its bidual, \cite[Section 2, Definition 3, Theorem 1]{Z90},
implies that for $\fY$ valued germs $\bf g$ 
the analytic continuation of ${\bf g}\abt$ is given by ${\bf g}_t\abt$, and this proves the lemma.
\end{proof} 
The following theorem generalizes Theorem 0.1:
\begin{thm} 
Suppose that a holomorphic Banach bundle $E\to\fB$ is endowed with a flat norm $p$ (Definition 6.1).
Assume that an $E|\fC$ valued holomorphic germ $\bf f$ at some $\fx\in\fC$  admits bounded analytic 
continuation along any path in $\fC$. Then the $E''$ valued germ ${\bf f}\abt$ admits
 analytic continuation along any path in $\fB$, with the same bound as the continuations within $\fC$.
\end{thm}

The proof will be given in section 9, after preliminary work in section 8.

 \section{The key construction}

 The key to the proof of Theorem 7.5 is a construction of a family of holomorphic maps of balls into $X$,
 with the help of which we can pull back analytic continuation along certain paths in $\fB$ to continuation in a
 Banach space. Endow $X$ 
with a metric $d$, then, as in section 2, metrize the space $B(S,X)$ by the metric
\begin{equation*}
d_S(\fx,\fy)=\sup \{d\big(\fx(s),\fy(s)\big):\, s\in S\}.
\end{equation*}  

\begin{lem}
Given $\fx\in C(S,X)$ and  $\varepsilon>0$, there is a $\delta>0$ with the following property. 
Suppose $S=\coprod_0^J S_j$ is a Borel partition, $S_0=A$ is closed, $a<b$ are real numbers, and
$[a,b]\ni t\mapsto\fy_t\in B(S,X)$ is a continuous path. Suppose furthermore that $\fy_t=\fx$ on 
$S_0$ and $\fy_t$ is
constant on $S_j$, $j=1,\dots,J$, $t\in T$. If $d_S(\fx,\fy_a)<\delta$, then there are a (finite rank) normed 
topological complex vector bundle $\pi:(V,|\,\,|)\to S$  (cf. section 4) with unit ball bundle 
$U=\{v\in V:|v|<1\}\to S$, a $\fv\in B_A(V)$, $\sup_S|\fv|<1$, and a continuous
$\Psi:[a,b]\times U\to X\times\bC$ such that for every $t\in[a,b]$ and $s\in S$
\newline\hspace*{1em} (a) $\Psi$ maps the fibers $\{t\}\times U_{s}$ biholomorphically  on open subsets
of  $X\times\bC$;
\newline\hspace*{1em} (b) the zero vector $0_{s}\in V_{s}$ satisfies $\Psi(t,0_{s})=\big(\fx(s),0\big)$;
\newline\hspace*{1em} (c) with $\text{pr}_X:X\times\bC\to X$ the projection, 
$\text{pr}_X\circ\Psi(t,\fv)=\fy_t$;
\newline\hspace*{1em} (d) $d\big(\fx(s),\text{pr}_X(\Psi(a,v))\big)< \varepsilon$ for $v\in U_{s}$.
\end{lem}

The proof takes some preparation.
Most of the construction will revolve around real analyticity.
Let $Y$ be a connected
real analytic manifold. In what follows, 
we will drop `real’ from `real analytic’. By a theorem of Grauert, $Y$ can be properly and analytically embedded 
in some Euclidean space \cite{G58}; accordingly, we will assume $Y\subset\bR^N$ is a closed analytic 
submanifold. For any $x\in\bR^N$ in some neighborhood $R$ of $Y$
 there is a unique nearest point $r(x)$ in $Y$; the map $r:R\to Y$ is an analytic retraction. 

Fix a topological space $T$ and a compact interval $I\subset\bR$. We will say that a map
$g:T\times I\to Y$ is partly analytic, $g\in C^{0\omega}(T\times I,Y)$, if $T\times I\subset T\times \bC$
has a neighborhood $O$ to which $g$ extends as a continuous function $\tilde g:O\to\bC^N$ whose
restrictions $\tilde g(t,\cdot)$ are holomorphic, $t\in T$. One can check that partial analyticity is independent
of which embedding $Y\subset\bR^N$ is chosen, and it does not even matter if $Y$ is closed in $\bR^N$, or
only in some open subset of $\bR^N$. It is clear that the composition of the partly analytic $g$ with an analytic
map $Y\to Y_1$ produces a partly analytic map.

\begin{lem}
Suppose $M,P$ are finite dimensional complex manifolds, and 
$g\in C^{0\omega}\big((T\times M)\times I,P\big)$ is such that $g(t,\cdot,q):M\to P$ is holomorphic
for each $t\in T$, $q\in I$. Then $T\times M\times I\subset T\times M\times\bC$ has a neighborhood
$O$ to which $g$ extends as a continuous $f:O\to P$, whose restrictions $f(t,\cdot,\cdot)$
are holomorphic for each $t\in T$. If $O_{t,z}=\{q\in\bC:(t,z,q)\in O\}$ are convex (that can always be achieved 
at the price of shrinking $O$), such extension $f$ is unique.
\end{lem}
\begin{proof} Assume first that $M$, $P$ are open subsets of some $\bC^m$, $\bC^p$.
With $w_j$,  $\zeta_h$ the complex coordinates on $\bC^p$, respectively $\bC^{2p}$, consider the
embedding $\alpha$ of $\bC^p$ (and so of $P$) into $\bR^{2p}\subset\bC^{2p}$, and its left inverse $\beta$,
\[
\alpha(w)= (\Ree w_1,\Imm w_1,\dots,\Ree w_p,\Imm w_p),\qquad 
\beta(\zeta)=(\zeta_1+i\zeta_2,\dots,\zeta_{2p-1}+i\zeta_{2p}).
\]
By definition, $\alpha\circ g$ extends to a continuous
$\tilde g:O\to\bC^N$  whose restrictions $\tilde g(t,z,\cdot)$ are holomorphic. It follows that,
possibly after shrinking $O$, $f=\beta\circ\tilde g:O\to P$ extends $g$ and
has holomorphic restrictions $f(t,z,\cdot)$; we can also arrange that the $O_{t,z}$ are convex.
That even the restrictions $f(t,\cdot,\cdot) $ are holomorphic will follow by Morera's and
Osgood's theorems once we show, denoting the coordinates on $\bC^m$ by $z_1,\dots,z_m$,
that  each $(t,z, q)\in O$ has a neighborhood $G\subset U$ such that
$\int_{\Gamma} f(t,z,q)\,dz_k=0$ whenever $\Gamma\subset M$ is a 
smooth loop  contained in a complex line, $\{t\}\times\Gamma\times\{q\}\subset G$, and $k=1,\dots,m$. 
But this is clear,
as the integral is a holomorphic function of $q$, vanishing when $q\in I$. 

The extension $f$ is unique by the
uniqueness theorem for holomorphic functions.

For a general $P$, cover $T\times M\times I$ by subsets $T_i\times M_i\times I_i$ 
that $g$ maps into coordinate neighborhoods on $P$, $T_i$ open,
$M_i\subset M$ coordinate neighborhoods, 
$I_i\subset I$ compact intervals. We have just constructed extensions
$f_i:O_i\to P$ to neighborhoods $O_i$ of $T_i\times M_i\times I_i$. By uniqueness the $f_i$ are
compatible and produce the extension $f$ claimed.
\end{proof}

 \begin{lem}
 If $T$ is compact and $L\subset I$ is finite, any
 $f\in C(T\times I,Y)$ can be approximated arbitrarily closely (in the uniform topology) by 
 $g\in C^{0\omega}(T\times I,Y)$ such that $f=g$ on $T\times L$.
 \end{lem}
 \begin{proof}
We will avail ourselves of the following tools. 
 There is a linear operator of interpolation, $\Lambda$, from the space $C(L,\bR^N)$ of all functions
 $L\to\bR^N$ to the space of polynomials $\bR\to\bR^N$ such that $(\Lambda\phi)|L=\phi$ for 
 $\phi\in C(L,\bR^N)$. Out of $\Lambda$ we construct 
 \[
 \bar \Lambda:C(T\times L,\bR^N)\to C^{0\omega}(T\times I,\bR^N),\qquad 
 (\bar \Lambda\varphi)(t,\cdot)=\Lambda\big(\varphi(t,\cdot)\big), 
 \]
 continuous between the uniform topologies on $C(T\times L,\bR^N)$ and
 $C^{0\omega}(T\times I,\bR^N)$.
 
Furthermore, any $\varphi\in C(T\times I,\bR^N)$ can be approximated uniformly by functions in
 $C^{0\omega}(T\times I,\bR^N)$. (For this latter, extend $\varphi$ to a compactly supported continuous
 $\tilde\varphi:T\times \bR\to\bR^N$. The approximating functions to $\varphi$ can be obtained by convolutions, for example, with the Poisson kernel:
 \[
 \varphi_\varepsilon(t,x)=\frac{\varepsilon}\pi\int_{\bR}\dfrac {\tilde\varphi(t,y)\,dy}{|x-y|^{2}+\varepsilon^{2}},
 \qquad (t,x)\in T\times I,\quad\varepsilon>0.
 \]
 Then $\varphi_\varepsilon\in C^{0\omega}(T\times I,\bR^N)$ tend uniformly to $\varphi$  as
 $\varepsilon\to 0$.)

 To construct $g$ of the lemma, first uniformly approximate $f$  by 
 $\phi\in C^{0\omega}(T\times I,\bR^N)$. Next let 
 $\psi=\phi+\bar \Lambda\big((f-\phi)|T\times L\big)\in C^{0\omega}(T\times I,\bR^N)$, so that 
 $\psi|T\times L=f|T\times L$. 
 If the approximation by $\phi$ was good enough, then $\phi$ and $\psi$ map $T\times I$ into the retract
 neighborhood $R$, and
 $g=r\circ \psi$ is the approximation of $f=r\circ f$ sought.
 \end{proof} 
    
 We return to the set up of  Lemma 8.1. Let
 $I\subset\bR$ be a compact interval and $q_0,\dots,q_J\in I$ distinct points.

  \begin{lem}
 Given $\fx\in C(S,X)$ and $\varepsilon>0$, there is a $\delta>0$ with the
 following property. 
 Suppose $S=\coprod_0^J S_j$ is a Borel partition, $S_0$ is closed, $a<b$ are real numbers, and
$[a,b]\ni t\mapsto\fy_t\in B(S,X)$ is a continuous path. Suppose furthermore that $\fy_t=\fx$ on $S_0$, 
and each $\fy_t$ is
constant on $S_j$, $j=1,\dots,J$. If $d_S(\fx,\fy_a)<\delta$, then  there is a 
$\psi\in C^{0\omega}\big(([a,b]\times S)\times I, X\big)$ such that
\newline\hspace*{1em} (a) $\psi(t,s,q_0)=\fx(s)$ if $t\in[a,b]$, $s\in S$;
 \newline\hspace*{1em} (b) $\psi(t,s,q_j)=\fy_t(s)$ if $t\in[a,b]$, $s\in S_j$, $j=1,\dots,J$;
  \newline\hspace*{1em} (c) $d\big(\fx(s),\psi(a,s,q)\big)< \varepsilon$ if $s\in S$, $q\in I$.
 \end{lem}
 \begin{proof}
 Embed $X$ as a closed real
 analytic submanifold of some $\bR^N$, with $r:R\to X$ an analytic retraction of a neighborhood 
 $R\subset\bR^N$. We can assume that the metric $d$ of $X$ is the restriction of the Euclidean metric; and
 also that $a=0$ and $\varepsilon<1$. Choose $\delta<\varepsilon/2$ so that the $\delta$--neighborhood of 
 \[
 \{x\in X: |x|\le1+\sup_S|\fx|\}\subset \bR^N
 \]
 is contained in the retract neighborhood $R$; and over this 
 $\delta$--neighborhood the retraction satisfies $|r-\id|<\varepsilon/2$. Given $\fy_t$, we will construct $\psi$
 in several steps.
 Consider 
\[ 
C=\big( S\times\{q_0\}\big)\cup\bigcup_{j=1}^J \big(\bar S_j\times\{q_j\}\big),\qquad
C_1=  [0,b]\times C,\qquad
C_2=\{0\}\times S\times I;
\]
$C_1,C_2$ are closed subsets of $[0,b]\times S\times I$. Define $\psi_1:C_1\to X$ by
\begin{equation*}
\psi_1(t,s,q_j)=\begin{cases}\fx(s) &\text{if }j=0,\quad s\in S\\
\fy_t(S_j) &\text{if }j\ge 1,\quad s\in\bar S_j.\end{cases}
\end{equation*}
In particular
\begin{equation} 
d\big(\fx(s),\psi_1(0,s,q)\big)\le d_S(\fx,\fy_0)<\delta\qquad\text{if}\quad (0,s,q)\in C_1\cap C_2.
\end{equation}

 Let $B_s\subset R$ denote the closed ball of radius $\delta$ about $\fx(s)$. Extend 
 $\psi_1|C_1\cap C_2$ to a continuous function $\psi_2:C_2\to\bR^N$, then for $(s,q)\in S\times I$ 
 define $\psi_3(0,s,q)\in B_s$ as the closest point to $\psi_2(0,s,q)$. Thus $\psi_3:C_2\to R$ is continuous
 and by (8.1) $\psi_1=\psi_2=\psi_3=r\circ\psi_3$ on $C_1\cap C_2$. Hence
 \begin{equation*}
 \psi_4=\begin{cases}\psi_1 &\text{on } C_1 \\ r\circ\psi_3 &\text{on } C_2\end{cases}
 \end{equation*}
 defines a continuous function $C_1\cup C_2\to X$, which extends to a continuous $\psi_5:G\to X$, with some neighborhood $G$ of $C_1\cup C_2$. Choose a neighborhood $G_1\subset S\times I$ of $C$ such that
 $[0,b]\times G_1\subset G$, and let a continuous  $\chi: S\times I\to[0,1]$ be supported in $G_1$,
 equal to $1$ on $C_1\cup C_2$. Then 
 \(\psi_6(t,s,q)=\psi_5\big(t\chi(s,q),s,q\big)\) defines a continuous function $[0,b]\times S\times I\to X$, that agrees
 with $\psi_1$ on $C_1$ and with $r\circ\psi_3$ on $C_2$. Requirements (a,b) of the lemma hold
 for $\psi_6$ by the definition of $\psi_1$, and (c) follows because $(0,s,q)\in C_2$
 when $s\in S$, $q\in I$, and so
 \begin{multline*}
 d\big(\fx(s),\psi_6(0,s,q)\big)=d\big(\fx(s),r(\psi_3(0,s,q))\big) \\
\le d\big(\fx(s),\psi_3(0,s,q)\big)+d\big(\psi_3(0,s,q),r(\psi_3(0,s,q))\big)<\delta+\varepsilon/2<\varepsilon.
 \end{multline*}
 A sufficiently good uniform approximation $\psi\in C^{0\omega}\big(([a,b]\times S)\times I,X\big)$ of $\psi_6$
 such that $\psi(t,s,q_j)=\psi_6(t,s,q_j)$ for all $t,s,j$, as guaranteed by Lemma 8.3, will then do. 
 \end{proof}
 
 \begin{proof}[Proof of Lemma 8.1]
 To start, let $D\subset X\times X$ be a neighborhood of the diagonal and $F:D\to TX$ as in Lemma 2.1, and
 $\exp$ as in (2.4). Possibly after shrinking $X$ we can arrange that $F$ and therefore $\exp$ are analytic. 
 
 Consider the holomorphic vector bundle $\Hm\to X\times X$ whose fibers are 
 $\Hm_{x,y}=\Hom(T_xX,T_yX)$, and the open subset $\Is\subset\Hm $ consisting of isomorphisms.
 The fiber bundle $p:\Is\to X\times X$ has a canonical analytic section, $A$, over $D$: this is so, because
 for fixed $x$ the tangent spaces to $T_xX$  are canonically isomorphic, and the
 differential of $F(x,\cdot)$ sets up isomorphisms between $T_xX$, resp. $T_yX$ on the one hand, and the
 tangent spaces to $T_xX$ at $F(x,x),F(x,y)$ on the other. $A(x,y)\in\Hom_{x,y}$ 
 is then the composition of these isomorphisms or their inverses.
 
 Let $I=[-1,1]$ and $q_0=0,q_1,\dots, q_J\in I$ distinct points. We will show that the $\delta$ that
 Lemma 8.4. provides works for Lemma 8.1 as well. Given $a,b,\fy_t$ as in Lemma 8.1,
let $\psi\in C^{0\omega}\big(([a,b]\times S)\times I,X\big)$ be the map that Lemma 8.4 constructs. With 
$\psi_0=\psi(a,\cdot,0):S\to X$ consider the sum of $\psi_0^*TX\to S$ and the trivial line bundle,
\[
\pi:V=(\psi_0^*TX)\times\bC\to S,
\]
 into which $(\psi_0^*TX)\times I$ is embedded. 

Let furthermore $\tilde\psi(t,s,q)=\psi_0(s)=\psi(a,s,0)$. The vector bundle $\tilde\psi^*TX$
is isomorphic to $\psi^*TX$, because $\psi,\tilde\psi:[a,b]\times S\times I\to X$ are homotopic. 
We use this isomorphism to construct a $C^{0\omega}$ map
\begin{equation} 
\theta:([a,b]\times S)\times I\to \Is, \qquad p\big(\theta(t,s,q)\big)=\big(\psi_0(s),\psi(t,s,q)\big),
\end{equation}
as follows. An isomorphism of the bundles translates to a continuous map $\theta_0$ as in (8.2), that, 
by Lemma 8.3,
can be approximated by $\theta_1\in C^{0\omega}\big(([a,b]\times S)\times I, \Is\big)$. 
Let $p\circ\theta_1=\alpha\times\beta$, with $\alpha,\beta:[a,b]\times S\times I\to X$ close to $\tilde\psi$, 
resp. $\psi$.  Then
\[
\theta=\big(A\circ(\beta\times\psi)\big)\,\theta_1\,\big(A\circ(\tilde\psi\times\alpha)\big)\in
C^{0\omega}\big(([a,b]\times S)\times I,\Is\big)
\]
satisfies (8.2).  (Above the values of $\theta$ are products of three linear transformations.) 

With $\varpi:\psi_0^*TX\to S$ the bundle projection and $\Xi\subset\psi_0^*TX$ a suitable neighborhood 
of the zero section,
\begin{equation} 
\phi:([a,b]\times \Xi)\times I\ni(t,\xi,q)\mapsto \big(\exp(\theta(t,\varpi\xi,q)\xi),q\big)\in X\times \bC
\end{equation}
defines a $C^{0\omega}$ map that is holomorphic along the fibers of $\varpi$. We have
\begin{equation} 
\phi(t,\xi_0,q)=\big(\psi(t,s,q),q\big)\qquad\text{when $\xi_0\in(\psi_0^*TX)_s$ is the zero vector,}
\end{equation}
since then $\theta(t,\varpi\xi_0,q)\xi_0\in T_{\psi(t,s,q)}X$ is the zero vector. 

Locally in the base, 
the bundle $\psi_0^*TX$ is a direct product, and $\Xi$ contains a neighborhood of the zero section
that is a product of a piece of $S$ with a complex manifold. Lemma 8.2 therefore gives, possibly after 
shrinking $\Xi$, an extension of $\phi$ to a neighborhood 
$O\subset[a,b]\times \Xi\times\bC\subset[a,b]\times V$ of $[a,b]\times \Xi\times I$. The extension
\[
\Phi=\Phi_1\times\Phi_2:O\to X\times\bC
\]
is continuous, and holomorphic along the fibers $O_{t,s}=O\cap(\{t\}\times V_s)$. By (8.3) 
$\Phi_2(t,\xi,q)= q$. We embed $S$ into
$\psi_0^*TX$ as the zero section, and claim that---possibly after shrinking $O$ to a smaller neighborhood of
$[a,b]\times S\times I\subset[a,b]\times V$---the restrictions $\Phi|O_{t,s}$ map biholomorphically on open 
subsets of $X\times\bC$.

Indeed, fix a zero vector $\xi_0\in T_{\psi_0(s)}X$ and $q\in I$. According to the splitting 
$V=(\psi_0^*TX)\times \bC$, the tangent space $T_{(t,\xi_0,q)}O_{t,s}$ also splits
\begin{equation} 
T_{(t,\xi_0,q)}O_{t,s}\approx T_{s}\psi_0^*TX\oplus T_{q}\bC.
\end{equation}
In the splitting (8.5) the differential at $(t,\xi_0,q)$ of $\Phi_1|O_{t,s}$ is $\exp_*\theta(t,s,q)\oplus 0$,
cf. (8.3), (8.4),
and of $\Phi_2|O_{t,s}$ it is $0\oplus\id_{T_q\bC}$. Since $\exp$ is biholomorphic on any $T_xX$,
it follows that ${(\Phi|O_{t,s})}_*$ is invertible at all points of $[a,b]\times S\times I$, hence also in 
a neighborhood in $[a,b]\times V$. As  $\Phi$ is injective on $\{(t,s)\}\times I$, a simple indirect 
argument shows that if 
$$
O\subset [a,b]\times V=[a,b]\times (\psi_0^*TX)\times \bC
$$ 
is sufficiently shrunk, all $\Phi|O_{t,s}$ will be biholomorphic.

We can take $O$ of form $[a,b]\times O'\times \Omega$, with $O'\subset \psi_0^*TX$ a neighborhood
of the zero section, and $\Omega\subset\bC$ a convex neighborhood of $I$. Let 
$\Delta\subset\bC$ be the unit disc
and $f:\Delta\to \Omega$ biholomorphic, that maps $0$ to $0$. Choose a norm 
$|\,\,|_1$ on $\psi_0^*TX$ whose unit ball bundle is contained in $O'$, and define a norm $|\,\,|$ on $V$ by
\[
|(\xi,\lambda)|=\max\big(|\xi|_1,|\lambda|\big),\qquad \xi\in \psi_0^*TX,\quad\lambda\in\bC,
\]
whose unit ball bundle $U$ is contained in $O'\times \Delta$. The formula
\begin{equation} 
\Psi(t,\xi,\lambda)=\Phi\big(t,\xi,f(\lambda)\big),\qquad 
t\in[a,b],\quad (\xi,\lambda)\in U,
\end{equation}
then defines a continuous $\Psi:[a,b]\times U\to X\times\bC$, whose restrictions to $\{t\}\times U_s$ are
biholomorphisms to open subsets of $X\times\bC$, as (a) claims. 

On $[a,b]\times S\times I$, by (8.4), pr$_X\circ\phi$ and pr$_X\circ\Phi$ agree with $\psi$. Therefore 
Lemma 8.4c implies that, if $O$ above is taken a sufficiently small neighborhood of  $[a,b]\times S\times I$,
then $d\big(\fx(s),\text{pr}_X(\Phi(a,v))\big)<\var$ for $v$ in the fiber $V_s\cap O$. Hence 
$d\big(\fx(s),\text{pr}_X(\Psi(a,v))\big)<\var$ when $v\in U_s$: (d) is also satisfied.

Finally, by (8.4) and by Lemma 8.4a,b (recall that $S\subset\psi_0^*TX$ and $S\times\{0\}\subset V$ are
the zero sections)
\begin{equation*}
\Phi(t,s,q_j)=\phi(t,s,q_j)=\big(\psi(t,s,q_j),q_j\big)=
\begin{cases}\big(\fx(s),0\big) &\text{ if } s\in S,\,j=0\\\big(\fy_t(s),q_j\big) &\text{ if } s\in S_j,\,j=1,\dots,J.
\end{cases}
\end{equation*}
Therefore if we define $\fv(s)=f^{-1}(q_j)$ when $s\in S_j$, $j=0,\dots J$, $\Psi$ will satisfy the remaining 
requirements (b,c) as well.
\end{proof}

\section{The extension theorem in mapping spaces} 

In this section we prove Theorem 7.5. Since the spaces $C(S,X,A,\fx_0)\subset B(S,X,A,\fx_0)$ do not change if
$\fx_0$ is replaced by any other $\fx\in C(S,X,A,\fx_0)$, we will take $\fx=\fx_0$. Recall the notion of basic 
neighborhoods $D_A(\fy)\subset \fB$, defined by a neighborhood $D\subset X\times X$ of the diagonal, see 
(5.1).  If $D$ is as in Lemma 2.1, then $D_A(\fy)$ 
is contractible. 

Given the $E|\fC$ valued germ $\bf f$ of Theorem 7.5, let $C\in\bR$ be such that any analytic continuation 
${\bf h}_t$  of $\bf f$ along any path $\fx_t$ in $\fC$ satisfies $p\big({\bf h}_t(\fx_t)\big)  \le C$.
Lemmas 9.1, 9.3 below continue ${\bf f}\abt$ along special
paths $\fy_t$ in $\fB$.
Fix $D$ above so that  $\bf f$ is the germ of a 
holomorphic section $f$ of $E|\fC\cap D_A(\fx)$ and $f$ has an Aron--Berner--type extension $f\abt$ to a section of
$E''|D_A(\fx)$. Then ${\bf f}\abt$ is the germ of $f\abt$.
\begin{lem} 
Consider a path $[0,1]\ni t\mapsto \fy_t\in \fB$ starting at $\fx$. Assume $a\in(0,1)$ 
is such that $\fy_t\in D_A(\fx)$ when $t\le a$. 
 Assume also that there are a normed topological vector bundle $\pi:(V,|\,\,|)\to S$
 with unit ball bundle $U\to S$, a $\fv\in B_A(V)$, $\sup_S|\fv|<1$, and a continuous
$\Psi:[a,1]\times U\to X$ such that for every  $t\in[a,1]$ and $s\in S$ 
\newline\hspace*{1em} (a) $\Psi$ maps the fibers $\{t\}\times U_{s}$ biholomorphically  on open subsets
of  $X$;
\newline\hspace*{1em} (b) the zero vector $0_{s}\in V_{s}$ satisfies $\Psi(t,0_{s})=\fx(s)$;
\newline\hspace*{1em} (c)  $\Psi(t,\fv)=\fy_t$;
\newline\hspace*{1em} (d) $\Psi(a,v)\in D^{\fx(s)}$ for $v\in U_{s}$, cf. (2.3).
\newline Then ${\bf f}\abt$ has an analytic continuation ${\bf g}_t$ along $\fy_t$, such that 
$p\big({\bf g}_t(\fy_t)\big)\le C$.
\end{lem}
Let $\Gamma\subset B_A(V)$ be the unit ball. 
\begin{lem}
For $t\in[a,1]$ the map
\begin{equation}
\Psi_t^\circ:\Gamma\ni\fu\mapsto\Psi(t,\fu)\in D_A(\fx),\quad\text{ cf. (d) above,}
\end{equation}
is biholomorphic on its image, an open subset of $D_A(\fx)$.
\end{lem}
\begin{proof}
Fix $t\in [a,1]$. First we claim that the map
\[
G: U\ni v\mapsto\big(\pi v,\Psi(t,v)\big)\in S\times X
\]
is injective and open. Injectivity follows from (a) of Lemma 9.1, since $G$ maps 
$U_s$ to $\{s\}\times X$ injectively. 
That $G$ is open is a local property. Accordingly, we can assume that $V\to S$ is trivial, $V=S\times\bC^r$.
 Suppose $\Psi(t,s_0,u_0)=z_0\in X$ with some 
$(s_0,u_0)\in U$, and let $B\subset\bC^r$ be a small compact neighborhood of $u_0$. The topological degree
$\deg\big(\Psi(t,s_0,\cdot),B,z_0\big)=1$. Hence $\deg\big(\Psi(t,s,\cdot),B,z\big)=1$ for $s,z$ close to
$s_0,z_0$; $\Psi(t,s,\cdot)$ attains the value $z$ in $ B$, 
and it follows that $G$ is indeed open. This
implies that $\Psi_t^\circ$ is open. It is also injective.

The inverse $\Phi$ of $G$ is a homeomorphism between an open $O\subset S\times X$ and
$U$, holomorphic on $O\cap\big(\{s\}\times X)$, and satisfies $\Psi\big(t,\Phi(s,z)\big)=z$,
$s\in S$, $z\in D^{\fx(s)}$. The map
\begin{equation}
\Phi^\circ:\Psi^\circ_t(\Gamma)\ni\fz\mapsto\Phi\circ(\id_S\times\fz)\in\Gamma
\end{equation}
is the inverse of $\Psi_t^\circ$, and both $\Phi^\circ,\Psi_t^\circ$ are holomorphic. To check the holomorphy
of $\Psi_t$, say, take a local coordinate $\psi_\fz$ on $\fB$ as in (5.2) with
$F:D\to TX$ of Lemma 2.1,
\[
\psi_\fz:D_A(\fz)\ni\fw\mapsto F\circ(\fz\times\fw)\in B_A\big(\fz^*TX).
\]
The formula $H(v)=F\big(\fz(\pi v),\Psi(t,v)\big)\in T_{\fz(\pi v)}X$ defines a map from a Borel
subset of $V$ to $\fz^*TX$ that satisfies the assumptions of Lemma 1.3. Therefore
\(
\psi_\fz\circ\Psi^\circ_t=F\circ(\fz\times\Psi_t^\circ),
\)
the restriction of $H^\circ$ of that lemma, is holomorphic on the open subset of $\Gamma$ where it is defined.
This means $\Psi_t^\circ$ is indeed
holomorphic. A similar argument proves for $\Phi^\circ$; therefore both are biholomophic.
\end{proof}

\begin{proof}[Proof of Lemma 9.1]
For values $t\le a$ the germs of $f\abt$ at $\fy_t$ supply the continuation ${\bf g}_t$. To continue ${\bf g}_a$ to values $t\ge a$, we will prove that with $\Psi^\circ_t$ of (9.1):
\newline\hspace*{2em} (1) ${\bf f}\abt$ is the germ of a holomorphic section $g_t$ of $E''|\Psi_t^\circ(\Gamma)$,
$\sup p''(g_t)\le C$;
\newline\hspace*{2em} (2) the germ of $g_a$ at $\fy_a$ is ${\bf g}_a$;
\newline\hspace*{2em} (3) the germs ${\bf g}_t$ of $g_t$ at $\fy_t$, $t\in[a,1]$, provide an analytic 
continuation of ${\bf g}_a$.

(1) Fix $t\in[a,1]$. The pullback  $\Psi_t^{\circ*}(E,p)=(\bar E,\bar p)\to\Gamma$ is a holomorphic 
Banach bundle with a flat norm $\bar p$. 
By Theorem 6.3 it is therefore trivial, and its (bounded) holomorphic sections can be viewed as (bounded) 
holomorphic functions with 
values in a fixed Banach space $(\fY,||\,\,||)$. Since the germ ${\bf k}={\bf f}\circ\Psi_t^\circ$ has 
bounded analytic continuation
along any path in $\Gamma\cap\fC$ (supplied by the pullback along  of $\Psi_t^\circ$ of the 
corresponding continuation of $\bf f$), it is the germ of a
bounded holomorphic $k:\Gamma\cap\fC\to\fY$. This $k$ has an Aron--Berner--type 
extension $k\abt$ to $\Gamma$ with $\sup p''(k\abt)\le C$,  constructed in section 4. Corollary 4.6 implies that 
the germ of
$k\abt$ at $0$ is ${\bf f}\abt\circ\Psi_t^\circ$, and this means that 
$g_t=k\abt\circ\Phi^\circ$ will do, cf. (9.2). 

(2) Since $\Psi_a^\circ(\Gamma)\subset D_A(\fx)$ is connected, and $f\abt$, $g_a$ share the same germ at 
$\fx$, namely 
${\bf f}\abt$, on $\Psi^\circ_a(\Gamma)$ we have $f\abt=g_a$. In particular, their germs at 
$\fy_a=\Psi^\circ_a(\fv)$ agree.

(3) Given $t\in[a,1]$, connect $\fx$ with $\fy_t$ by a path in $\Psi_t^\circ(\Gamma)$. By continuity of $\Psi$,
for $\tau\in[a,1]$ close to $t$, this path is contained in $\Psi_\tau^\circ(\Gamma)$. Since the germs of
$g_t,g_\tau$ at $\fx$ agree, $g_t,g_\tau$ agree in a neighborhood of the path, hence in a neighborhood
of $\fy_t$. But this means that the germs
${\bf g}_t$ depend continuously on $t$, and they continue ${\bf g}_a$ analytically. This completes the proof.
\end{proof}
Again we metrize $X$ with a metric $d$, then $d_S$ of (2.2) metrizes $\fB$.

\begin{lem}
There is a $\delta>0$ with the following property. Let
$S=\coprod_{j=0}^J S_j$ be a partition in Borel sets, $S_0=A$. Suppose that $[0,1]\ni t\mapsto\fy_t\in \fB$ is a path
starting at $\fx$, and $a\in(0,1)$ is such that 
\newline\hspace*{1em} (a) $d_S(\fx,\fy_t)<\delta$ when $t\in[0,a]$;
\newline\hspace*{1em} (b) $\fy_t|S_j$ is constant when $t\in[a,1]$, $j=1,2,\dots, J$.
\newline
Then ${\bf f}\abt$ has an analytic continuation ${\bf g}_t$ along $\fy_t$, such that 
$p\big({\bf g}_t(\fy_t)\big)\le C$.
\end{lem}
\begin{proof} 
With $X^+=X\times \bC$ and pr$_X:X^+=X\times\bC\to X$ the projection, let 
 $D^+=(\text{pr}_X\times\text{pr}_X)^{-1}D\subset X^+\times X^+$, a neighborhood of the diagonal, 
 $\fx^+=\fx\times0\in C(S,X^+)$,
\[
\fB^+=\{\fz\in B(S,X^+):\fz|A=\fx^+|A\}, 
\]
and $\fC^+$ the component of $\fx^+$ in $\fB^+\cap C(S,X^+)$.  Choose 
$\var>0$ so that $D(\fx)$ contains the $d_S$ metric ball of radius $\var$  about $\fx$.
Given $\fx$ and $\var$, the $\delta>0$ provided by Lemma 8.1 will do.

Indeed, if (a), (b) of our lemma
hold, choose a 
normed vector bundle $(V,|\,\,|)\to S$ with unit ball bundle $U\subset V$, $\fv$ in the unit ball 
$\Gamma\subset B_A(V)$, 
and a map $\Psi:[a,1]\times U\to X^+$  as guaranteed by Lemma 8.1. 
Consider the projection
 $P:\fB^+\to\fB$ given by $P(\fz)=\text{pr}_X\circ\fz$; it is holomorphic and maps $\fC^+$ to 
$\fC$. The pullback $f\circ P$ is a holomorphic section of $P^*E|\big(\fC^+\cap D^+(\fx^+)\big)$, whose
germ at $\fx^+$ we denote $\bf h$. Now $\bf h$ analytically  continues along any path $r$ in
$\fC^+$, the continuations provided by the lifts of the continuations along $P\circ r$; these continuations are bounded,
with the same bound $C$ as for $\bf f$. 
Let us set $\fz_t=\Psi(t,\fv)\in \fB^+$, $t\in[0,1]$. We are in the situation of Lemma 9.1, with $X$ replaced by $X^+$,
$\fB$ by $\fB^+$, $\fy_t$ by $\fz_t$. By the lemma ${\bf h}\abt$ has an analytic continuation 
${\bf k}_t$ along $\fz_t$. 
It remains to show that the $P^*E''$ valued germs ${\bf k}_t$ are pullbacks of germs ${\bf g}_t$ that form 
the analytic continuation of ${\bf f}\abt$ along $\fy_t$.

Let us call a holomorphic vector field $\xi$ on $\fB^+$ vertical if it is tangential to the fibers $P^{-1}(\fy)$,
$\fy\in\fB$. Consider a Banach space $\fY$ and a $\fY$ valued holomorphic germ $\bf k$ at some 
$\fz\in\fB^+$.
It is of form ${\bf k}={\bf g}\circ P$ with a $\fY$ valued holomorphic germ $\bf g$ at $P(\fz)$ if and only if
the derivatives $\xi{\bf k}=0$ for all vertical vector fields $\xi$. Similarly, if $\bf k$ is a $P^*E''$ valued germ at
$\fz\in\fB^+$, we can compute its vertical derivatives $\xi\bf k$, since the restriction of (a representative
of) $\bf k$ to a fiber $P^{-1}(\fu)$ takes values in the fixed Banach space $E''_{P(\fu)}$;
and $\bf k$ is a pullback if and only if $\xi{\bf k}=0$ for all vertical $\xi$.

For the analytic continuation ${\bf k}_t$ of ${\bf h}\abt={\bf k}_0$ constructed above, $\xi{\bf k}_0=0$, because
by Corollary 4.6 ${\bf k}_0$ is the pullback of ${\bf f}\abt$ along $P$. Hence its analytic continuation 
$\xi{\bf k}_t$ also vanishes, and  
${\bf k}_t$ is the pullback of some germ ${\bf g}_t$ at $\fy_t$. By Lemma 7.2 ${\bf g}_t$ depend continuously on $t$, 
hence supply the analytic continuation of ${\bf f}\abt$.
\end{proof}

We will reduce continuation along arbitrary paths to continuation along paths as in Lemma 9.3 in the 
following way.
By a partition we will mean a decomposition $S=\coprod_{j=0}^J S_j$ into Borel sets such that $S_0=A$,
$S_j\neq\emptyset$, $j=1,\dots,J$. Given such a partition, denoted
 $\Pi$, let 
 \begin{equation}
 \fB_\Pi=\{\fy\in\fB: \fy|S_j=\text{const},\text{ for }j=1,\dots,J\}.
 \end{equation}
 If $\Pi'$ refines $\Pi$ (notation: $\Pi'\succ \Pi$), then $\fB_{\Pi'}\supset\fB_{\Pi}$.
 
 \begin{lem}
 (a) $\fB_\Pi\subset\fB$ is a closed submanifold, biholomorphic to $X^J$;
 
 (b) If $\Pi_0$ is a partition, then $\bigcup_{\Pi\succ\Pi_0}\fB_\Pi$ is dense in $\fB$;
 
 (c) Let $c\in \bR$, $\fU\subset\fB$ open, $K\to\fU$ a holomorphic Banach bundle with a flat norm $p$, and
 $\Pi_0$ a partition. Suppose for each $\Pi\succ\Pi_0$ we are given a holomorphic section $h_\Pi$
 of $K|\fB_\Pi\cap\fU$ such that $p(h_\Pi)\le c$, and $\Pi'\succ\Pi$ implies $h_{\Pi'}=h_\Pi$ on $\fB_\Pi\cap\fU$. 
 Then there is a unique holomorphic section $h$ of $K|\fU$ such that $h|\fB_\Pi\cap\fU=h_\Pi$.
 
 (d) Let $K\to\fB$ be a holomorphic Banach bundle with a flat norm. If a $K$ valued holomorphic germ 
 $\bf h$ at some $\fz\in\fB_{\Pi_0}$ has
 analytic continuation along any path in $\fB_\Pi$ for $\Pi\succ\Pi_0$, and these continuations have the same bound 
 $c$, then $\bf h$ has an analytic continuation along any path in $\fB$, bounded by $c$.
 \end{lem} 
 
 \begin{proof}
 (a) $\fB_\Pi$ is clearly a closed subset. If $\fy\in\fB_\Pi$, local coordinates $\psi_\fy$ as in (5.2) send the
  neighborhood $D_A(\fy)$ to an open $\fV\subset B_A(\fy^*TX)$. For all $j\ge 1$ the fibers of $\fy^*TX$ over
  $S_j$ are the same tangent spaces to $X$, hence one can talk about sections of $\fy^*TX$ being constant over
  $S_j$.  Those bounded Borel sections of $\fy^*TX$ that vanish on $S_0$ and are constant over each $S_j$, $j\ge 1$,
  form a finite dimensional subspace
  \[
  B_\Pi(\fy^*TX)\subset B_A(\fy^*TX),
  \]
  and the local coordinate $\psi_\fy$ sends $\fB_\Pi\cap D_A(\fy)$ to $B_\Pi(\fy^*TX)\cap\fV$. 
  Therefore $\fB_\Pi$ is 
  indeed a submanifold of $\fB$. One can then check that the map that sends $(x_1,\dots,x_J)\in X^J$ to 
  $\fz\in\fB_\Pi$ defined by $\fz(s)=x_j$ if $s\in S_j$, $j=1,\dots,J$, is biholomorphic.
  
  (b) Let $\fy\in\fB$ and $\var>0$. To find a partition $\Pi\succ\Pi_0$ and $\fz\in\fB_\Pi$ in the $\var$ neighborhood of 
  $\fy$, cover the closure of $\fy(S)\subset X$ by finitely many balls of radius $\var$, choose $\Pi\succ \Pi_0$,
  $S=\coprod_{j=0}^J S_j$, so that each $\fy(S_j)$ is contained in one of those balls, centered at,
  say, $y_j\in X$, $j\ge 1$. Then $\fz\in\fB_\Pi$ given by 
  \begin{equation*} 
  \fz(s)=\begin{cases}\fx(s)& \text{ when } s\in S_0=A \\   y_j& \text{ when } s\in S_j, \quad 
  i=1,\dots, J\end{cases}
  \end{equation*}
  will do.
  
  (c) We can assume that $\fU$ is  contained in a coordinate neighborhood about some $\fy\in\fB$ and
   that $(K,p)$ is trivial over it. Then instead of sections one can talk about
  Banach space valued functions. The local coordinate $\psi_\fy$ maps
  $\fU$ to a neighborhood $\fW\subset B_A(\fy^*TX)$ of the zero section, and $\fB_\Pi\cap\fU$ into
  $B_\Pi(\fy^*TX)$ from the proof of part (a). A hermitian metric 
  on $X$ induces a norm on $B_A(\fy^*TX)$, and we can assume that $\fW$ is the unit ball.
  
  Accordingly, the claim will be proved once we show the following. Let $\fX$, $(\fY,||\,\,||)$ be Banach spaces,
  $\fW\subset\fX$ the unit ball, and $c\in\bR$. Suppose for each partition $\Pi\succ\Pi_0$ we are 
  given a closed subspace $\fX_\Pi\subset \fX$ and $h_\Pi\in\cO(\fX_\Pi\cap \fW,\fY)$ with 
  the following properties.
  First, $||h_\Pi||\le c$; second, $\fX_{\Pi'}\supset\fX_\Pi$ and $h_{\Pi'}|\fX_{\Pi}\cap \fW=h_\Pi$ 
  if $\Pi'\succ\Pi$; and 
  finally, $\fZ=\bigcup_{\Pi\succ\Pi_0}\fX_\Pi$ is dense in $\fX$. Then there is a unique $h\in\cO(\fW,\fY)$ 
  such that
  $h|\fX_\Pi\cap\fW=h_\Pi$ and $||h||\le c$. (Of course, partitions here are irrelevant, $\Pi$ could denote elements of an
  arbitrary directed set.)
  
  The density of $\fZ$ implies that $h$ is unique. To construct $h$, define $g:\fZ\cap\fW\to\fY$ by 
  $g(x)=h_\Pi(x)$ if $x\in\fX_\Pi\cap \fW$. Cauchy estimates for the derivatives of $h_\Pi$ imply that
  on a concentric ball $\fW_r\subset\fW$ of radius $r$ the Lipschitz constant of $h_\Pi$ is $\le c/(1-r)$, hence
  the same holds for $g$. In particular, if $x\in\fW$ and a sequence $x_i\in\fZ\cap \fW$ converges to $x$,
  then $g(x_i)$ form a Cauchy sequence. Define $h(x)=\lim_i g(x_i)$; the limit is independent of the
  sequence $x_i\to x$. The function $h$ constructed is continuous and its norm
  is bounded by $c$. To complete the
  proof we need to check that for any $x,y\in\fX$ and $\Lambda=\{\lambda\in\bC:x+\lambda y\in \fW\}$,
  the function $f:\Lambda\ni\lambda\mapsto h(x+\lambda y)$
  is holomorphic. Let $\Lambda_0\subset \Lambda$ be open and relatively compact. Choose sequences 
  $\Pi_i\succ \Pi_0$, and $x_i,y_i\in\fX_{\Pi_i} $ convergent to $x,y$. Then 
  $f_i(\lambda)=h(x_i+\lambda y_i)=h_{\Pi_i}(x_i+\lambda y_i)$ is holomorphic on $\Lambda_0$ when
  $i$ is sufficiently large. Since $||f_i||\le c$ and $f_i\to f$ on $\Lambda_0$, it follows that $f$ is holomorphic,
  and therefore so is $h$.
  
  (d) Let $r:[a,b]\to \fB$ be a path, $\bf h$ a germ at $r(a)=\fz$. Let $\fW_t\subset D_A\big(r(t)\big)$ 
  be neighborhoods of $r(t)$ whose image under the local coordinate 
  $\psi_{r(t)}$ is convex, $r(a)\in\fW_t$ for $t$ in a neighborhood of $a$,
  and $\bf h$ is the germ of a holomorphic section $h^0$ of $K|\fW_a$. 
  Then $\fB_\Pi\cap\fW_t$ for each $\Pi$ is contractible  or  empty. 
   Choose $a=t_0<t_1<\dots<t_k=b$ so that  $r[t_{j-1},t_j]\subset\fW_{t_j}$ when $1\le j\le k$. 
  By recurrence we construct holomorphic sections $h^j$ of $K|\fW_{t_j}$ that
  define the analytic continuation of $\bf h$ along $r$. If $j\ge 1$ and $h^i$ for $i<j$ has been
  constructed, choose a partition $\Pi_1\succ\Pi_0$ and $\fw$ in the connected component of
  $r(t_{j-1})$ in $\fB_{\Pi_1}\cap\fW_{t_{j-1}}\cap\fW_{t_j}$. The analytic continuations of the germ of
  $h^{j-1}$ at $\fw$ along paths in $\fB_\Pi\cap\fW_{t_j}$, $\Pi\succ\Pi_1$, define holomorphic sections
  $h_\Pi$ of $K|\fB_\Pi\cap \fW_{t_j}$.  Since these sections satisfy the conditions in (c), they are 
  restrictions of a holomorphic section $h^j$ of $K|\fW_{t_j}$.  By the uniqueness part of (c) the germs of 
  $h^{j-1}$ and $h^j$ agree at $\fw$, therefore in the connected component of $\fw$ in 
  $\fW_{t_{j-1}}\cap\fW_{t_j}$. In particular, $h^{j-1}$ and $h^j$ agree near $r(t_{j-1})$. Thus the collection
  $h^j$ indeed gives rise to an analytic continuation of  $\bf h$ along $r$.
  \end{proof}
  
  \begin{proof}[Proof of Theorem 7.5]  It suffices to construct continuations along paths 
  $[0,1]\ni t\mapsto\fx_t\in\fB$, $\fx_0=\fx$.
  Let $\delta>0$ as in Lemma 9.3 and $a\in(0,1/2)$ be such that the $\delta$--neighborhood $\fU\subset\fB$
  of  $\fx$ is contained in $D_A(\fx)$ and contains $\fx_t$ when $t\in[0,2a]$. 
   Choose a partition $S=\coprod_{i=0}^I T_i$, denoted $\Pi_0$, such that $\fB_{\Pi_0}$ intersects 
  $\fU$  (Lemma 9.4b), say, $\fy\in \fB_{\Pi_0}\cap \fU$. Since the 
   image of $\fU$ under the local coordinate $\psi_\fx$ is convex, there is a path 
  $[0,2a]\ni t\mapsto \fy_t\in\fU$ from $\fx$ to $\fx_{2a}$ such that $\fy_a=\fy$; extend it to 
  $2a\le t\le 1$ by $\fy_t=\fx_t$. The germs of $f\abt$ provide  continuation of ${\bf f}\abt$ along both
  $\fx_t,\fy_t$, $t\in[0,2a]$. These continuations give the same result at $\fx_{2a}$, that we denote
   ${\bf g}$. Let $\bf h$ be the germ of $f\abt$ at $\fy_a$. 
  To produce analytic continuation of
  ${\bf f}\abt$ along  $\fx_t$, $t\in[0,1]$, it suffices to show continuation of ${\bf g}$ along $\fx_t$,
  $t\in[2a,1]$, or of $\bf h$ along $\fy_t$, $t\in[a,1]$. But Lemma 9.3 guarantees that this latter
  continuation exists along paths $\fy'_t$ for which
  $\fy'_t\in \fB_\Pi$ when $t\in [a,1]$; here $\Pi\succ\Pi_0$ is arbitrary. Hence, by Lemma 9.4d, 
  the continuation exists along $\fy_t$ as well. This completes the proof.
  \end{proof}
  
  \section{The Monodromy theorem} 
  
  In this section we prove the following generalization of Theorem 0.2. Consider mapping spaces
  $\fB,\fC$ determined by a compact Hausdorff space $S$, a complex manifold $X$, a continuous
  $\fx\in C(S,X)$ and a closed $A\subset S$, as before.   
  \begin{thm}
  Suppose $X$ is simply connected, $E\to\fB$ is a holomorphic Banach bundle with
  a flat norm $p$,
and an $E|\fC$ valued holomorphic germ $\bf f$ admits bounded analytic 
  continuation along any path in $\fC$. Then $\bf f$ is the germ of a holomorphic section $f$ of $E|\fC$. If,
  furthermore, $X$ is compact, then $f$ is a horizontal section (Definition 6.4).
  \end{thm}
  
  Fix a partition $S=\coprod_{j=0}^J S_j$ as in section 9, denote it $\Pi_0$, 
  and let $\bar\fx\in\fB_{\Pi_0}$, see (9.3). Given $\nu=0,1,\dots$, inclusions induce 
  a directed system of homotopy
  groups/sets and homomorphisms
  \begin{equation}
  \pi_\nu(\fB_\Pi,\bar\fx)\to\pi_\nu(\fB_{P},\bar\fx),\qquad P\succ\Pi\succ\Pi_0.
  \end{equation}
 
 \begin{lem}
 $\pi_\nu (\fB,\bar\fx) $ is isomorphic to the direct limit of the system (10.1). Hence, if $X$ is connected (our 
 standing assumption),
 resp. simply connected, then so is $\fB$.
 \end{lem}
 \begin{proof}
 We need to show that any continuous map $(S^\nu,*)\to(\fB,\bar\fx)$, and any homotopy between two
 such maps, can be deformed to a map, respectively homotopy, into $\fB_\Pi$ with some partition 
 $\Pi\succ\Pi_0$; and that any homotopy
 in $(\fB,\bar\fx)$ between continuous maps $f_1,f_2:(S^\nu,*)\to (\fB_\Pi,\bar\fx)$ can be deformed,
 fixing $f_1,f_2$, to a homotopy in some $(\fB_P,\bar\fx)$, $P\succ\Pi$. 
 Once this done, we can define a homomorphism $\pi_\nu(\fB,\bar\fx)\to\varinjlim \pi_\nu(\fB_\Pi,\bar\fx)$
 by first sending the class in $\pi_\nu(\fB,\bar\fx)$ of an $f:(S^\nu,*)\to(\fB,\bar\fx)$ to an
 $f_1:(S^\nu,*)\to(\fB_\Pi,\bar\fx)$ homotopic to $f$, and then to
 the element of $\varinjlim \pi_\nu(\fB_P,\bar\fx)$
 that the class $[f_1]\in\pi_\nu(\fB_\Pi,\bar\fx)$ induces. This homomorphism is automatically bijective.  
 
 The deformation statements  are special cases of the following: Let
 $T$ be a compact Hausdorff space, $C\subset T$ closed. 
 For any continuous $\phi:(T,C)\to(\fB,\fB_\Pi)$ there is a continuous deformation 
 $\phi_\tau: (T,C)\to(\fB,\fB_\Pi)$, $0\le\tau\le 1$, such that $\phi_0=\phi$, $\phi_\tau|C=\phi|C$ for all $\tau$,
 and $\phi_1(T)\subset \fB_P$ with some $P\succ\Pi$. 
  To verify this, we will use $D\subset X\times X$, $D^x$, $F:D\to TX$ of Lemma 2.1, and $\exp:TX\to X$
  of (2.4).  Since $\phi(t)\in\fB$ is a continuous function of $t\in T$,
 and each $\phi(t)(S)\subset X$ has compact closure, there is a compact $L\subset X$ that contains all
 $\phi(t)(S)$.  Cover $L$ by finitely many coordinate neighborhoods $D^{x_i}$, and choose a partition $P\succ\Pi$, given by 
 $S=\coprod_{i=0}^I R_i$, such that $\phi(t)(R_i)\subset D^{x_i}$ for
 every $t\in T$ and $i=1,\dots, I$. If $s_i\in R_i$, $i\ge 1$, then
 \begin{equation*}
 \phi_\tau(t)(s)=\begin{cases}\phi(t)(s)=\fx(s) &\text{ if } s\in R_0=A\\ 
 \exp\big\{(1-\tau)F\big(x_i,\phi(t)(s)\big)+\tau F\big(x_i,\phi(t)(s_i)\big)\big\}&\text{ if } s\in R_i,\quad i=1,\dots, I
 \end{cases}
 \end{equation*}
  is the deformation sought.
 
 The second statement of the lemma follows because by Lemma 9.4a the  spaces $\fB_\Pi$ are
 homeomorphic to Cartesian powers of $X$, hence for (simply) connected $X$ they themselves are 
 (simply) connected.
 \end{proof} 
 
 \begin{proof}[Proof of Theorem 10.1]
 Since  $\fB$ is connected and simply connected by Lemma 10.2, $(E,p)$ is trivial. Again, instead of
 sections of $(E,p)$ we can talk about functions with values in some fixed Banach space $(\fY,||\,\,||)$. Accordingly,
 we take $\bf f$ to be a $\fY$ valued germ. By Theorem 7.5 ${\bf f}\abt$ admits bounded analytic continuation
 along any path in $\fB$; again because $\fB$ is simply connected, this implies that ${\bf f}\abt$ is the
 germ of a holomorphic $g:\fB\to\fY''$. Hence $\bf f$ is the germ of the holomorphic function $f=g|\fC$. Since
 its germ $\bf f$ takes values in $\fY\subset\fY''$, all the values that $f$ takes are in $\fY$.
 
 If, in addition, $X$ is compact, then the submanifolds $\fB_\Pi\approx X^J$ are also compact 
 (and connected), 
 hence the restrictions $g|\fB_\Pi$ are constant. As $\bigcup_\Pi\fB_\Pi$ is dense in $\fB$, Lemma 9.4b,
 this implies that $g$ itself is constant, and therefore $f$ is constant.
 \end{proof}

 \section{Examples}
 
 In this section we have collected examples that show that some assumptions in Theorems 7.5, 10.1 cannot
 be dispensed with, or that the conclusions cannot be strengthened. We keep notation used in sections
 7, 10.
 
 \begin{ex} 
 There are mapping spaces $\fC,\fB$ and a bounded holomorphic function $f:\fC\to\bC$ with
 germ $\bf f$ at some $\fx_0$ such  that ${\bf f}\abt$ is not the germ of a single valued holomorphic function
 $g:\fB\to\bC$.
 \end{ex}
 Let $r\in(1,\infty)$, $X=\{x\in\bC:1/r<|x|<r\}$ an annulus, $S=[0,1]$, $A=\emptyset$, and
 $\fx_0\equiv 1$. Thus $\fB,\fC$ are open sets in the Banach spaces $B(S)$, $C(S)$. If $\fx\in\fC$,
  define $f(\fx)=\sqrt{\fx(1)/\fx(0)}$, using the branch of the square root obtained by  
 continuation from $\sqrt 1=1$ along the path $[0,1]\ni t\mapsto \fx(t)/\fx(0)$. Let $\bf f$ be the germ
 of $f$ at $\fx_0$. In a neighborhood of $\fx_0$, we can write
 $f(\fx)=\sqrt{\phi(\fx)/\psi(\fx)}$ with holomorphic functions
 \begin{equation}
 \phi(\fx)=\fx(1),\quad \psi(\fx)=\fx(0)
 \end{equation}
 on $C(S)$ and the principal branch of square root. Here $\phi,\psi$ are linear functions, 
 and their extensions $\phi\abt, \psi\abt$ to $B(S)$ are given by the same formula (11.1). 
 Since the Aron--Berner
 extension is compatible with multiplication, it follows that ${\bf f}\abt$ is the germ of
 \[
 g(\fx)=\sqrt{\fx(1)/\fx(0)},\qquad \fx\in\fB\text{ in a neighborhood of }\fx_0.
 \]
 But ${\bf f}\abt$ has no single valued continuation e.g. along the path $[0,1]\ni t\to\fx_t\in\fB$,
 \begin{equation*}
 \fx_t(s)=\begin{cases}e^{2\pi it} &\text{ if }s=1\\1 &\text{ if } s\neq 1.\end{cases}
 \end{equation*}
 
 \begin{ex} 
 There are mapping spaces $\fB,\fC$ and (unbounded) $f\in\cO(\fC)$ with germ
 $\bf f$ at some $\fx_0\in\fC$ such that ${\bf f}\abt$ cannot be continued along some path in $\fB$.
 \end{ex}
 This is a minor reformulation of a result of Dineen. Let $X=\bC$, $S=\{0,1/n:n\in\bN\}\subset\bR$ with the
 topology inherited from $\bR$, $A=\emptyset$, and $\fx_0\equiv 0$. Thus $\fC$ is biholomorphic to
 $c$, the space of convergent sequences of complex numbers, and $\fB\supset\fC$ is biholomorphic to 
 $c''=l^\infty\supset c$. In particular, $\fB$ is simply connected, whence if a germ $\bf g$ on it has analytic
 continuation along any path, then it is the germ of a holomorphic function on $\fB$. Since 
 according to Dineen and Josefson \cite{D71, J78} there
 are holomorphic functions $f$ on $c$ that do not extend to $l^\infty$, our claim follows. In fact,
 not only ${\bf f}\abt$ cannot be continued analytically along some path, but no extension $\bf g$ of ${\bf f}$ 
 to a germ on $\fB$ has analytic continuation along all paths in $\fB$.
 
 \begin{ex} 
 (a) Unless $X$ is contractible, it has a mapping space $\fC$ with a flat normed holomorphic Banach
 bundle $(E,p)\to\fC$, and there is an $E$ valued holomorphic germ at some $\fx_0\in\fC$ that has bounded 
 analytic continuation along any path in $\fC$, yet does not extend to a holomorphic section of $E$. If $X$ is
 simply connected, then $E$ can be chosen a line bundle.
 
 (b)  If a Stein manifold $X$ is not contractible, it has a corresponding mapping space $\fC$, and there is a
 $\bC$ valued holomorphic germ $\bf f$ at some $\fx_0\in\fC$ that continues analytically along any path in
 $\fC$, yet $\bf f$ does not extend to a holomorphic function $\fC\to  \bC$.
 \end{ex}
 Both in (a) and (b) $X$ can be simply connected. Thus (b) shows that in Theorem 10.1 the assumption 
 that the analytic continuations are bounded cannot be simply dropped,
 and (a) shows that even boundedness is not sufficient unless it is known that $(E,p)$ extends to a flat
 bundle over $\fB$.---If $X$ is contractible, then its mapping spaces are also contractible, and the phenomena
 described in Example 11.3 will not occur.
 
 If $X$ is not contractible, according to Whitehead it has a nontrivial homotopy group $\pi_\nu(X,x_0)$, $\nu\ge 1$
 \cite[Theorem 1]{W49}. Let $S=S^{\nu-1}$, $A\subset S$ a singleton, $\fx_0\equiv x_0$, and
 $\fC=C(S,X,A,\fx_0)$. Since
 $\pi_1(\fC,\fx_0)\approx\pi_\nu(X,x_0)\neq 0$, there is a nontrivial covering $\pi:C\to \fC$ by a connected complex
 Banach manifold $C$. For $\fx\in\fC$ let $C_\fx=\pi^{-1}(\fx)$, and let  $E_\fx$ denote the Hilbert space of functions 
 $e:C_\fx\to\bC$ with $l^2$ norm  $p(e)=\big(\sum_{v\in C_\fx}|e(v)|^2\big)^{1/2}<\infty$. On
 $E=\coprod_{ \fx\in\fC} E_\fx$ the local trivializations of $C\to\fC$ induce the structure of a holomorphic Banach bundle 
 over $\fC$, with flat norm $p$. 
 
 The pullback $\pi^*(E,p)\to C$ is trivial, and has a horizontal section $g$ (Definition 6.4)
 whose value at $y\in C_\fx$ is the 
 delta function $e\in E_\fx$ given by  $e(y)=1$, $e(v)=0$ if $v\neq y$. The germ $\bf g$ of $g$ at any point in 
 $C_{\fx_0}$ pushes forward to a holomorphic $E$ valued germ $\bf f$ at $\fx_0$; the pushforwards of other
 germs of $g$ provide bounded analytic continuations of $\bf f$ along any path in $\fC$; but $\bf f$ is not the germ
 of a section $f$ of $E$. If it were, $g$ would be the pullback of $f$, and the singleton supports of the functions 
 $f(\fx):C_\fx\to\bC$ would provide a global section of the covering $C\to\fC$, which is impossible.

 If $X$ is simply connected, then the above construction can be modified to produce a line bundle $E\to\fC$.
 Namely, $\nu\ge 2$ then, whence $G=\pi_1(\fC,\fx_0)\approx\pi_\nu(X,x_0)$
 is Abelian. By Pontryagin duality it admits a nontrivial representation $\rho:G\to \text{U}(1)$, 
 \cite[Section 37]{P66}. This induces a flat line bundle $(E,p)\to\fC$, obtained from the trivial line bundle
 $\tilde\fC\times\bC\to\tilde\fC$ over the universal cover of $\fC$, endowed with the trivial norm = Hermitian
 metric, by factoring out by the action
 \[
 \tilde\fC\times \bC\ni(\fy,\zeta)\mapsto\big(\gamma\fy,\rho(\gamma)\zeta\big)\in\tilde\fC\times\bC, 
 \qquad \gamma\in G,
 \]
 ($G$ acts on $\tilde\fC$ by deck transformations). As $\rho$ is nontrivial, $(E,p)$, as a flat line bundle, is
 nontrivial, either---although, as a holomorphic line bundle $E$ may very well be trivial. The section
 $\tilde\fC\ni y\mapsto(y,1)\in\tilde\fC\times\bC$ induces a multivalued horizontal section of $(E,p)$, 
 any of whose germs at $\fx_0$ is the $\bf f$ sought.
 
 A further modification is possible when $X$ is Stein. In this case the universal cover $\tilde\fC$ can be 
 embedded in a Banach space. To see this, we first embed $X$ holomorphically and properly into some Euclidean
 space $\bC^N$, and arrange that in the embedding $x_0$ is sent to $0$. This exhibits $\fC$ as a closed
 complex direct submanifold\footnote{A subset $M$ of a complex Banach manifold $N$ 
 is a direct submanifold if 
 for each $q\in N$ there are a neighborhood $U\subset N$ of $q$, a complemented subspace $A\subset B$
 in a Banach space $B$, and an open $V\subset B$ such that the pairs $(U,U\cap M)$ and $(V,V\cap A)$ are
 biholomorphic.} 
  of the Banach space $C_A(S)^{\oplus N}$. The latter space is isomorphic to a finite
 codimensional subspace of $C(S)^{\oplus N}\approx C[0,1]^{\oplus N}$, hence has a Schauder basis,
 according to Miljutin and Schauder \cite{M66,S27}. All this implies that $\tilde\fC$ can be 
 embedded in a Banach space as a closed direct manifold. This is proved in \cite[Theorem 3.7]{L10}, except the
 assumption there is that $\fC$ should be embedded in a Banach space with unconditional basis. However, the 
 unconditionality of the basis there was used only to be able to quote certain earlier results of cohomological
 nature by Patyi and by this author. All those results follow 
 from cohomology vanishing in \cite[Theorem 9.1]{LP07}, one of
 whose assumptions (``plurisubharmonic domination'') could at the time be proved only when unconditional bases
 existed. In the meantime, however, Patyi extended plurisubharmonic domination to spaces with a Schauder basis
 \cite{P11,P15}. This means that embeddibility of the universal (or any) cover, \cite[Theorem 3.7]{L10}, holds under
 the assumption of Schauder, rather than unconditional, basis. This takes care of the embedding of our $\tilde \fC$.
 
 Once the embedding granted, for any pair of distinct points 
 $a,b\in\tilde\fC$ there is a $g\in\cO(\tilde\fC)$ that takes
 different values at $a,b$. If we choose $a,b$ both lie above $\fx_0$, then the pushforward $\bf f$  
 to $\fC$ of the germ
 of $g$ at $a$ continues analytically along any path in $\fC$, but those continuations do not define a single valued
 function on $\fC$.

\end{document}